\documentclass[11pt,twoside, reqno]{amsart}

\usepackage[latin1]{inputenc}
\usepackage[OT1]{fontenc}
\usepackage{amsmath}
\usepackage{amsthm}
\usepackage{amssymb}
\usepackage[all]{xy}
\usepackage{bbm}
\usepackage{amscd}
\usepackage{a4wide}
\usepackage{enumitem}
\usepackage{mathtools}

\setenumerate[0]{label=\upshape(\roman*)}

\DeclareMathOperator{\Spec}{\mathsf{Spec}}

\DeclareMathOperator{\pr}{\mathsf{pr}}

\DeclareMathOperator{\id}{\mathsf{id}}

\DeclareMathOperator{\Hom}{\mathsf{Hom}}

\DeclareMathOperator{\sHom}{\mathcal{H}\textit{om}}
\DeclareMathOperator{\sEnd}{\mathcal{E}\textit{nd}}
\DeclareMathOperator{\Ext}{\mathsf{Ext}}
\DeclareMathOperator{\End}{End}
\DeclareMathOperator{\ext}{\mathsf{ext}}
\DeclareMathOperator{\Coh}{\mathsf{Coh}}

\DeclareMathOperator{\sExt}{\mathcal{E}\textit{xt}}
\DeclareMathOperator{\Ho}{\mathsf H}
\DeclareMathOperator{\ho}{\mathsf h}

\DeclareMathOperator{\gr}{\mathsf{gr}}

\DeclareMathOperator{\supp}{\mathsf{supp}}

\let\deg\relax
\DeclareMathOperator{\deg}{\mathsf{deg}}

\let\mod\relax
\DeclareMathOperator{\mod}{\mathsf{mod}}

\let\det\relax
\DeclareMathOperator{\det}{\mathsf{det}}

\newcommand{\red}{\mathsf{red}}

\newcommand{\sing}{\mathsf{sing}}

\newcommand{\Sing}{\operatorname{Sing}}

\DeclareMathOperator{\Pic}{\mathsf{Pic}}
\DeclareMathOperator{\APic}{\mathsf{APic}}

\newcommand{\cI}{{\mathcal I}}

\newcommand{\ko}{{\mathcal O}}

\newcommand{\IC}{\mathbb{C}}

\newcommand{\IN}{\mathbb{N}}
\newcommand{\IP}{\mathbb{P}}
\newcommand{\IQ}{\mathbb{Q}}

\newcommand{\IZ}{\mathbb{Z}}

\newcommand{\IG}{\mathbb{G}}

\newcommand{\wJ}{\widetilde J}

\DeclareMathOperator{\Jac}{\mathsf{Jac}}
\DeclareMathOperator{\pure}{\mathsf{pure}}
\DeclareMathOperator{\Alb}{\mathsf{Alb}}

\makeatletter
\newcommand{\leqnomode}{\tagsleft@true}
\newcommand{\reqnomode}{\tagsleft@false}
\makeatother

\let\ker\relax
\DeclareMathOperator{\ker}{\mathsf{ker}}

\let\dim\relax
\DeclareMathOperator{\dim}{\mathsf{dim}}

\newcommand{\cF}{\mathcal F}

\newcommand{\cE}{\mathcal E}
\newcommand{\cG}{\mathcal G}

\newcommand{\cM}{\mathcal M}
\newcommand{\cN}{\mathcal N}

\newcommand{\cD}{\mathcal D}

\newcommand{\cP}{\mathcal P}
\newcommand{\cQ}{\mathcal Q}

\newcommand{\cL}{\mathcal L}

\newcommand{\fm}{\mathfrak m}

\newcommand{\reg}{\mathcal O}

\newcommand{\eps}{\varepsilon}

\renewcommand{\theta}{\vartheta}
\renewcommand{\rho}{\varrho}
\renewcommand{\phi}{\varphi}
\renewcommand{\_}{\underline{\,\,\,\,}}

\usepackage{enumitem}
\setlist[enumerate,1]{label={(\roman*)}}
\setlist[enumerate,2]{label={(\alph*)},ref={(\alph*)}}

\usepackage[usenames,dvipsnames]{xcolor}
\usepackage{hyperref}
\hypersetup{
    colorlinks,
    linkcolor={red!50!black},
    citecolor={blue!50!black},
    urlcolor={blue!80!black},
    linktoc=all 
}
\usepackage{aliascnt} % to get hyperref's \autoref to work properly with counters

\newtheorem{theorem}{Theorem}[section]

  \newaliascnt{conjecture}{theorem}
\newtheorem{conjecture}[conjecture]{Conjecture}
 \aliascntresetthe{conjecture}

  \newaliascnt{proposition}{theorem}
  \newtheorem{prop}[proposition]{Proposition}
  \aliascntresetthe{proposition}

  \newaliascnt{lemma}{theorem}
  \newtheorem{lemma}[lemma]{Lemma}
  \aliascntresetthe{lemma}

  \newaliascnt{corollary}{theorem}
  \newtheorem{cor}[corollary]{Corollary}
  \aliascntresetthe{corollary}

    \newaliascnt{assumption}{theorem}
  \newtheorem{assumption}[assumption]{Assumption}
  \aliascntresetthe{assumption}

\theoremstyle{definition}
  \newaliascnt{notation}{theorem}
  
  \aliascntresetthe{notation}

  \newaliascnt{definition}{theorem}
  \newtheorem{definition}[definition]{Definition}
  \aliascntresetthe{definition}

  \newaliascnt{remark}{theorem}
  \newtheorem{remark}[remark]{Remark}
  \aliascntresetthe{remark}

  \newaliascnt{condition}{theorem}
  
  \aliascntresetthe{condition}

  \newaliascnt{question}{theorem}
  
  \aliascntresetthe{question}

  \newaliascnt{example}{theorem}
  \newtheorem{example}[example]{Example}
  \aliascntresetthe{example}

\newcommand{\lra}{\longrightarrow}

\usepackage{tikz-cd}
\usepackage[root radius = .07cm, edge length = 0.7 cm]{dynkin-diagrams}

\newlist{prooflist}{description}{1}
\setlist[prooflist]{font=\normalfont \itshape, labelindent = \parindent, leftmargin = 0pt}

\newcommand{\isom}{\cong}
\newcommand{\tensor}{\otimes}

\begin{document}

\title{Compactified Jacobians of Extended ADE curves and Lagrangian Fibrations}
\author{Adam Czapli\'nski, Andreas Krug, Manfred Lehn, and S\"onke Rollenske}

\maketitle

\begin{abstract}
We observe that general reducible curves  in sufficiently positive linear systems on K3 surfaces are of a form that generalises Kodaira's classification of singular elliptic fibres and thus call them extended ADE curves.

On such a curve $C$, we describe a compactified Jacobian and show that its components reflect the intersection graph of $C$. This extends known results when $C$ is reduced, but new  difficulties arise when $C$ is non-reduced.  As an application, we get an  explicit description of  general singular fibres of certain Lagrangian fibrations of Beauville-Mukai type. 
\end{abstract}

\section{Introduction}
In this article we want to discuss the geometry of compactified Jacobian varieties 
associated to certain reducible or even non-reduced projective curves over a field
of characteristic $0$. Compactified Jacobian varieties have a long history. If 
$C$ is a smooth projective curve, the Picard group $\Pic^0(C)$ parameterising 
invertible sheaves of degree $0$ is an abelian variety of dimension equal to the 
genus $g(C)$, and Jacobian variety is used as a synonym for any of the connected 
components of $\Pic(C)$. If $C$ is an integral curve with singularities, the 
Picard group $\Pic(C)$ is still a group scheme of finite type, but no longer projective,
as line bundles can degenerate into torsion free sheaves that fail to be locally
free at one or more of the singular points of $C$. As the name suggests, a compactified 
Jacobian is an appropriate compactification of a component of $\Pic(C)$. 
For a historical overview and more details we refer to the paper by Altman and Kleiman \cite{AlKl1980} and references therein.

The problem becomes more difficult if one allows $C$ to be a non-integral curve, in particular if the curve has a  non-reduced scheme structure.
The Picard scheme of such curves has been described in full generality by Bosch, L\"utkebohmert 
and Raynaud \cite{Neron}. A compactified Jacobian for such curves can be
defined as, and this is the view adopted in this article, the moduli space of
pure sheaves on $C$ that are semi-stable with respect to an appropriate polarisation and have the same length as $\ko_C$ at each generic point of $C$. 
For non-integral curves, it can happen that (components of) the Picard scheme form an open part of the moduli space of these sheaves, but that this part is no longer dense. We will see that this happens in our context. 

In this paper we assume $C$ to be a projective curve in some ambient smooth surface
$X$ of the form $C=\sum_{i} m_iC_i$ with smooth components $C_i$ and certain
 multiplicities $m_i$. The dual graph $\Gamma$ is assumed to be an affine
simply-laced Dynkin graph, see  \autoref{fig: dynkin}, and the multiplicity vector $m=(m_i)\in \IZ^{\ell}$, $\ell
=|\Gamma|$, is to be the minimal positive integral null-vector of the semi-negative Cartan 
matrix $S$ associated to $\Gamma$. 
We assume that all multiple components are rational with self-intersection $-2$. We call a curve $C$ meeting these requirements an \emph{extended ADE-curve of type $\Gamma$}.

\begin{figure}[ht]
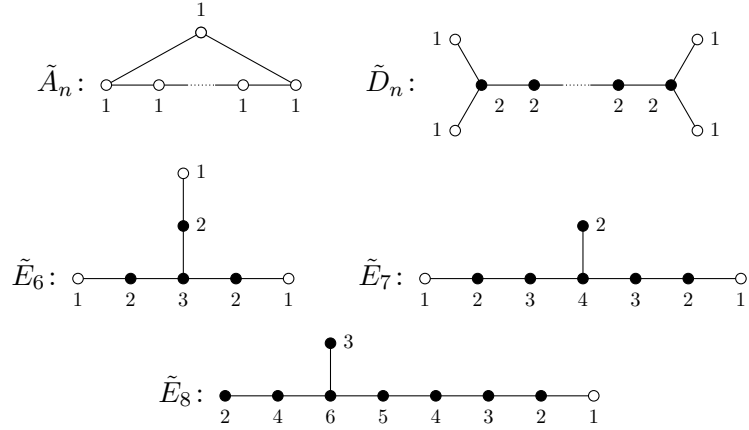

 \caption{The extended Dynkin graphs }\label{fig: dynkin}
 \begin{gather*}
   \tilde A_n\colon\dynkin[labels = {1,1,1,1,1}] A[1]{oo.oo}
   \qquad
\tilde D_n\colon \dynkin[labels = {1, 1, 2, 2, 2, 2, 1, 1}] D[1]{o**.**oo}
\\
\tilde E_6\colon
\dynkin[labels ={ 1,1,2, 2, 3, 2, 1}] E[1]{o****o}
\qquad 
\tilde E_7\colon
\dynkin[labels = {1,2, 2,3,4,3,2,1}] E[1]{******o}
% \qquad
\\
\tilde E_8\colon
\dynkin[labels = {1,2,3,  4, 6, 5, 4, 3, 2}] E[1]8
 \end{gather*}
\end{figure}

If $\Gamma$ is of type $\tilde A_{n}$, all 
multiplicities are $1$, hence the curve $C$ is reduced.
In \cite{LM--Simpson}, L\'{o}pez-Mart\'{\i}n classifies stable sheaves on reduced singular fibres of elliptic fibrations, in particular on $\tilde A_n$-configurations of smooth rational curves; for a comparison to our results see \autoref{subsect:LM}.
In the study of the non-reduced cases $\tilde D_n$ and $\tilde E_n$ several new difficulties arise.

Let $\cM_\chi(C)$ denote the Simpson moduli space of semistable sheaves $F$ with
Euler characteristic $\chi$ on $C$ such that $F$ has length $m_i$ at the generic 
points $\eta_i\in C_i$ (with appropriately chosen polarisation and conditions on 
$\chi$ to be discussed in the main text, where a sheaf satisfying our requirement on the lengths at the generic points is called a sheaf \emph{of type $m$}). 

The following is a concise reformulation of a large part of the results obtained in \autoref{sect:ADE}.
\begin{theorem}\label{thm: easy classification}
Let $H$ be a polarisation as in \autoref{ass:H}.
 Then every semi-stable sheaf in the compactified Jacobian $\cM_\chi(C)$ is already stable and of the form 
 \[L(x):= \sHom(\cI, L),\]
 where $\cI\subset \ko_C$ is the ideal sheaf of a closed point $x$ and $L$ is a line bundle in a fixed component of the Picard group, the component depending on $H$ and $\chi$.

 The sheaf $L(x)$ is a line bundle if and only if $x$ is a smooth point of $C$. Otherwise, $x$ is the unique point where $L(x)$  fails to be locally free.
 \end{theorem}

Our second main result describes the geometry of the compactified Jacobian $\cM_\chi(C)$. The following is deduced from a more detailed and technical description in  \autoref{thm: full moduli}. 
\begin{theorem}\label{thm: easy moduli}
Let $C=\sum_im_iC_i$ be an extended ADE curve of arbitrary type $\Gamma\in \{\tilde A_n,\tilde D_n, \tilde E_6,\tilde E_7,\tilde E_8\}$, and let $J:=\prod_i\Jac(C_i)$ denote the product of the Jacobians of the components of $C$. Then $\cM_\chi(C)_{\red}$ consists of components $\{Y_v\}_{v\in V}$ which are $\IP^1$-bundles over $J$. They intersect transversally in sections, and their intersection graph is again $\Gamma$.
\end{theorem}
For $\Gamma=\tilde A_n$, the compactified Jacobian is reduced, so the above already gives a description of $\cM_\chi(C)$ itself. For $\Gamma\in \{\tilde D_n, \tilde E_6,\tilde E_7,\tilde E_8\}$, those components $Y_v$ which correspond to multiple components of the curve $C$ are again non-reduced, but we do not compute their multiplicities. Our expectation, however, is that the multiplicity of $Y_v$ is equal to the multiplicity of $C_v$ in $C$; compare \autoref{Conj:multi}.

The motivation to consider curves of this specific form comes from our interest 
in determining the structure of generic singular fibres in Lagrangian 
fibrations $\cM_\chi(K3)\to \IP^n$ of moduli of semistable sheaves on K3 surfaces. Indeed, extended ADE curves occur as generic non-integral members of linear systems on K3 surfaces; see \autoref{sect: linear systems} and \autoref{subsect:family} for some more details on this.
The point of departure was the construction of a K3 surface as a minimal resolution of the singularities of a double cover of the plane branched along a reduced but possibly reducible singular sextic (with at worst ADE singularities) and the explicit classification of the singular fibers of the associated Beauville-Mukai system $\cM_\chi(K3)\to \IP^5=|\reg(2)|$, see Czapli\'nski \cite{Cz2018}. 

Let us sketch the structure of the article.  In \autoref{sect:generalitiesonshonc} we discuss some generalities concerning sheaves on curves. The next two sections are the technical heart of the work. In \autoref{sect:ADE} we completely classify stable sheaves on an extended ADE curve $C$ and in the following \autoref{sect:modulispacedescr} we give a description of the moduli space $\cM_\chi(C)$. In the last \autoref{sect:furtherremarks} we discuss, among other things, some codimension one strata in linear systems on K3 surfaces, Beauville-Mukai systems, characteristic cycles, and work of other authors related to our results.

\subsection*{Acknowledgements}
We would like to thank Andreas Knutsen for some advice concerning linear systems on K3 surfaces.

\section{Generalities concerning sheaves on curves}\label{sect:generalitiesonshonc}

\subsection{Picard varieties of curves}

In this subsection, a \emph{curve} is a purely one-dimensional scheme of finite type over $\IC$. In the later subsections, we will make stronger assumptions on our curves.

\begin{prop}\label{prop:Picgeneral}
 Let $C$ be a connected and projective curve with irreducible components $C_1,\dots, C_\ell$ (here, we mean the reduction of the components). Then, the Picard functor of $C$ is represented by a scheme $\Pic(C)$. Its connected component $\Pic^0(C)$ containing the point corresponding to the trivial line bundle consists of those line bundles whose restriction to every $C_i$ is of degree $0$. Let $C_i'\to C_i$ be the normalisation. Then, pull-back along the morphism $\coprod_{i=1}^\ell C_i'\to C$ gives an exact sequence of algebraic groups
\[
1\to G\to \Pic^0(C)\to \prod_{i=1}^\ell \Pic^0(C_i')\to 1,
\]
where $G$ is a smooth connected algebraic group of dimension $\ho^1(\reg_C)-\sum_{i=1}^\ell \ho^1(\reg_{C_i'})$.
\end{prop}
\begin{proof}
See Thm.\ 8.2.3 and Cor.\ 9.2.11 \& 9.2.13 of \cite{Neron}.
\end{proof}

\begin{cor}\label{cor:Piciso}
 If $\ho^1(\reg_C)=\sum_{i=1}^\ell \ho^1(\reg_{C_i'})$, then the natural map $\Pic^0(C)\to \prod_{i=1}^\ell \Pic^0(C_i')$ is an isomorphism.
\end{cor}

\subsection{Duals of ideal sheaves of singular points}

In this subsection, let $C$ be a (not necessarily projective) Gorenstein curve. 
By \cite[Prop.\ 1.6]{Har--gen}  the Gorenstein condition implies that a sheaf $F$ on $C$ is reflexive if and only if it is purely one-dimensional, so we do not lose information by taking duals.

For a point $x\in C$ we denote the dual of its ideal sheaf $\cI_x=\cI_{x\hookrightarrow C}$ by 
\[
 \reg_C(x):=\cI_{x}^\vee\,.
\]
Furthermore, given a line bundle $L\in \Pic(C)$, we write 
\[
L(x):=L\otimes \reg_C(x)\cong\sHom(\cI_{x}, L)\,.
\]

\begin{lemma}\label{lem:Oxseq}
\begin{enumerate}
\item For every $x\in C$, the sheaf $\reg_C(x)$ is purely one-dimensional, and there is a   
non-split short exact sequence 
\begin{equation}\label{eq:Cxses}
 0\to \reg_C\to \reg_C(x)\to \reg_x\to 0\,.
\end{equation} 
\item Conversely, every purely one-dimensional sheaf $F$ fitting into a short exact sequence
 \[
  0\to \reg_C\to F\to \reg_x\to 0
 \]
is already of the form $F\cong \reg_C(x)$. 
\end{enumerate}

\end{lemma}
\begin{proof}
Part (i) is \cite[Prop.\ 2.8 \& 2.10]{Har--gen}.
By the purity of $F$, the sequence of part (ii) 
\[
  0\to \reg_C\to F\to \reg_x\to 0
 \]
is non-splitting. By Serre-duality on a Gorenstein curve, $\ext^1(\reg_x,\reg_C)=1$. Hence, there is only one isomorphism class of sheaves which fit into such a non-split exact sequence. Hence, by part (i), $F\cong \reg_C(x)$.
\end{proof}

\begin{lemma}\label{lem:Cxsing}
 The sheaf $\reg_C(x)$ is a line bundle if and only if $x$ is a smooth point of $C$.
\end{lemma}
\begin{proof}
 A smooth point is a Cartier divisor. Hence, the associated sheaf is a line bundle.

 If $x$ is a singular point, we have $\ext^1(\reg_x,\reg_x)=\dim T_{C,x}>1$. Applying $\Hom(\_, \reg_x)$ to \eqref{eq:Cxses} gives an exact sequence 
 \[
  \IC=\Hom(\reg_C,\reg_x)\to \Ext^1(\reg_x,\reg_x)\to \Ext^1(\reg_C(x),\reg_x)\to 0\,.
 \]
This yields $\Ext^1(\reg_C(x),\reg_x)\neq 0$ which cannot happen for a line bundle. 
\end{proof}

\begin{lemma}\label{lem:lbcrit}
 Let $F$ be a purely one-dimensional sheaf on $C$. If there is a point $x\in C_{\mathsf{sing}}$ such that $F_{\mid C\setminus\{x\}}$ is a line bundle, and $F_{\mid U}\cong \reg_U(x)$ for some open neighbourhood $x\in U\subset C$, then there exists an $L\in \Pic(C)$ such that $F\cong L(x)$.  
\end{lemma}

\begin{proof}
This follows from \cite[Prop.\ 2.12]{Har--gen} noting that $\APic(C)=\Pic(C)$ for a curve.
\end{proof}

For two types of singularities $x\in C$, we need a more detailed description of $\reg_C(x)$, given by the following two lemmas.

\begin{lemma}\label{lem:End1}
Let $m\geq 2$ and $C=mD$ where $D$ is a smooth divisor in a smooth surface $X$, and let $x\in C$.
\begin{enumerate}
 \item The natural map $\sEnd_{\reg_C}(\cI_x)\to \reg_C(x)$, given by post-composition of endomorphisms with the embedding $\cI_x\to \reg_C$, is an isomorphism.
\item The $\reg_C$-algebra $\sEnd_{\reg_C}(\cI_x)$ is commutative and satisfies $\sEnd_{\reg_C}(\cI_x)_{\mathsf{red}}\cong \reg_{C_{\mathsf{red}}}$.
 \item Let $E\subset X$ be another smooth divisor intersecting $D$ transversally in $x$. Let $\zeta=mD\cap E$ be the scheme-theoretic intersection. Then, there is no surjective $\reg_C$-linear morphism $\reg_C(x)\to \reg_\zeta$.  
\end{enumerate}
\end{lemma}

\begin{proof}
 As the assertions are local, we can assume that $C=\Spec A$, with $A=\frac{\IC[u,v]}{(v^m)}$, and $x=V(\fm)$ with $\fm=(u,v)$. Then $\reg_C(x)=\Hom_A(\fm, A)$ is spanned, as an $A$-module, by the inclusion $\id\colon \fm\to A$ together with the homomorphism
 \[
  \phi\colon\quad u\mapsto v^{m-1}\quad,\quad v\mapsto 0\,.
 \]
Indeed, one checks that $\phi$ is well defined, and the sequence $0\to \reg_C\to \reg_C(x)\to \reg_x\to 0$ 
shows that any homomorphism which is not a multiple of $\id$ suffices to generate $\reg_C(x)$. As both, $\id$ and $\phi$ map $\fm$ to itself, part (i) follows. The relations
\[
 v\phi=0\quad,\quad u\phi= v^{m-1}\id\quad,\quad \phi^2=0
\]
give an isomorphism of $A$-algebras
\[
 \End_A(\fm)\cong \frac{\IC[u,v,\phi]}{(v^m,v\phi, u\phi-v^{m-1},\phi^2)}\,,
\]
which proves part (ii). For part (iii), we may assume that $\zeta=\Spec(B)$ with $B=\frac{\IC[u,v]}{(u,v^m)}$. 
Let $\alpha\colon \End_A(\fm)\to B$ be $A$-linear. The relation $v\phi=0$ shows that $\alpha(\phi)\in (v^{m-1})$. Furthermore, the relation $u\phi= v^{m-1}\id$ shows that 
\[
v^{m-1}\alpha(\id)=u\alpha(\phi)=0\,, 
\]
hence $\alpha(\id)\in (v)$. In summary, the image of $\alpha$ is contained in $(v)$, which means that $\alpha$ is not surjective.
\end{proof}

\begin{lemma}\label{lem:End2}
Let $D, E$ be two smooth divisors in a smooth surface $X$ intersecting transversally in a point $x\in X$.
Let $C:=mD+nE$ for some $m,n\in \IN$. 
\begin{enumerate}
 \item The natural map $\sEnd_{\reg_C}(\cI_x)\to \reg_C(x)$ is an isomorphism.
 \item The $\reg_C$-algebra $\sEnd_{\reg_C}(\cI_x)$ is commutative. 
 \item If $(m,n)\neq(1,1)$, we have $\sEnd_{\reg_C}(\cI_x)_{\mathsf{red}}\cong \reg_{C_{\mathsf{red}}}$. 
 \item If $(m,n)=(1,1)$, we have $\sEnd_{\reg_C}(\cI_x)\cong \reg_{D}\times \reg_E$. 
\end{enumerate}
\end{lemma}

\begin{proof}
We proceed analogously to the proof of \autoref{lem:End1}. We assume that $C=\Spec A$, with $A=\frac{\IC[u,v]}{(u^nv^m)}$, and $x=V(\fm)$ with $\fm=(u,v)$. Let us assume without loss of generality that $m\ge n$. Then $\reg_C(x)=\Hom_A(\fm, A)$ is spanned, as an $A$-module, by the inclusion $\id\colon \fm\to A$ together with the homomorphism
\[
  \phi\colon\quad u\mapsto u^nv^{m-1}\quad,\quad v\mapsto 0\,.
 \]
As both, $\id$ and $\phi$ map $\fm$ to itself, part (i) follows. As $\phi$ commutes with $\id$ and all its multiples, we get part (ii). 

If $(m,n)\neq (1,1)$, then $\phi$ is nilpotent. This gives part (iii). If $m=1=n$, then $\phi$ is not nilpotent, but $\phi^2=\phi$. We get an isomorphism 
\[
 \End_A(\fm)\cong \frac{\IC[u,v,\phi]}{(uv,v\phi, u\phi-u,\phi^2-\phi)}= \frac{\IC[u,v,\phi]}{(u,\phi)\cdot(v,\phi-1)} 
\]
which proves (iv).
\end{proof}

\begin{remark}\label{rem:End}
By \cite[Prop.\ 1.6]{Har--gen}, the sheaf $\cI_x$ is reflexive. It follows that $\sEnd(\cI_x)\cong \sEnd\bigl(\reg_C(x)\bigr)$.  
\end{remark}

\subsection{Pure sheaves on divisorial curves}\label{subsect:purediv}
In this subsection, on a smooth projective surface $X$ we consider possibly reducible and non-reduced curves. The curves are still assumed to be Gorenstein, hence without embedded points. We identify such a curve $C$ with the corresponding Weil divisor and write
\begin{equation*}
C = \sum_{i = 0}^ \ell m_i C_i \subset X, 
\end{equation*}
where  $C_0,\dots, C_\ell$ are the  irreducible components of $C$ and  $m_0,\dots, m_\ell$ their multiplicities. We assume that all the components $C_i$ are smooth.

Most of the results in this section are true in a more general context, but we did not feel the need to strive for maximal generality.

The following is an easy but useful observation.
\begin{lemma}\label{lem: decomposition sequence}
For any decomposition $C = A + B$ there are two short exact sequences
 \begin{align*}
  0 \to \ko_A(-B) \to \ko_C \to \ko_B\to 0\,,\\
  0 \to \ko_B(-A) \to \ko_C \to \ko_A\to 0\,.
 \end{align*}
The second map in each sequence is the restriction, and the two sequences are dual in the sense that they arise from each other by applying $\sHom_{\reg_C}(\_,\reg_C)$.
\end{lemma}

\begin{definition}
 Let $F$ be a coherent sheaf on a curve with one-dimensional support. We write
 \[
 \pure(F):=F/\text{(zero-dimensional torsion)}
 \]
for the maximal purely one-dimensional quotient of $F$. 
\end{definition}

\begin{lemma}\label{lem:quotientsarerestrictions}
 Let $F$ be a purely one-dimensional sheaf on a curve $C$ which is generically a line bundle. Then, every purely one-dimensional quotient $F\twoheadrightarrow F''$ is of the form $F''=\pure(F_{\mid Z})$ for some subcurve $Z\subset C$. 
\end{lemma}

Here, by a subcurve $Z\subset C$, we mean a subdivisor $Z=\sum_{i=0}^\ell n_iC_i$ with $0\le n_i\le m_i$.

\begin{proof}
 Let $U\subset C$ be a dense open subset with $F_{\mid U}\cong \reg_U$. Then, there is a subcurve $W\subset U$ such that $F_{\mid U}\twoheadrightarrow F''_{\mid U}$ is isomorphic to the restriction map 
$\reg_U\twoheadrightarrow \reg_W$. Let $Z\subset C$ be the unique subdivisor such that $Z\cap U=W$. Let $F':=\ker(F\twoheadrightarrow F'')$. The composition $F'\to F\to \pure(F_{\mid Z})$ is zero over the dense open set $U$. As $\pure(F_{\mid Z})$ is purely one-dimensional, we get that the composition is zero. Hence, we get a morphism of quotients 
\[
\begin{tikzcd}
& F \dar{\id}\rar & F'' \dar\rar &0\\
& F  \rar & \pure(F_{\mid Z}) \rar  &0\,.
\end{tikzcd}
\]
The right vertical map is surjective and an isomorphism over $U$. As $F''$ is purely one-dimensional, it is an isomorphism.
\end{proof}

From now on, we also assume that all components $C_i$ of $C$ are smooth, that at most two components of $C$ meet in one point, and that the components of $C$ intersect transversally. That is, at each intersection point $x\in C_i \cap C_j$ the intersection subscheme $\zeta_x:=m_iC_i\cap m_jC_j$ satisfies
\begin{equation}\label{eq:zeta}
\reg_{\zeta_x} \isom \IC[u,v]/(u^{m_i}, v^{m_j})\,.
\end{equation} 

We will study sheaves on $C$ by means of a canonically associated short exact sequence given as follows.

\begin{lemma}\label{lem:Fses}
Let $F$ be a sheaf on $C$ which is pure of dimension 1. For each component
$C_j$ define
\begin{equation*}
F_j:=\pure(F_{|m_jC_j})\,.
\end{equation*}
Then there is an associated exact sequence
\begin{equation*}%\label{eq:exactsequenceofsheavesrelatingFandFj}
0\lra F\lra \bigoplus_{j} F_j\lra \bigoplus_x T_x\lra 0,
\end{equation*}
where the first direct sum runs through all components of $C$, the second direct sum runs through all intersection points $x$ of components of $C$ and $T_x$ is a zero-dimensional sheaf supported in $x$ such that all components $F_j\to T_x$ with $x\in C_j$ are surjective.
\end{lemma}

\begin{proof}
 The map $F\to \bigoplus_{j} F_j$, whose components are the restriction maps, is an isomorphism over $U\subset C$, the complement of the intersection points. Hence, the map is injective, as $F$ is pure, and its cokernel $\bigoplus_x T_x$ is supported in the intersection points. 
 
 Let $x$ be one of the intersection points. By our assumption, only two components, say $C_i$ and $C_j$, meet in $x$.
Hence, locally near $x$, we have the following square
\begin{equation*}%\label{eq:Fjcommutative}
\begin{tikzcd}
F_i \arrow[r] 
& T_x  \\
F \arrow[u,twoheadrightarrow] \arrow[r,twoheadrightarrow]
&  F_j\arrow[u]\,,
\end{tikzcd}
\end{equation*}
which is commutative up to a sign.
It follows that the images of the two maps $F_i\to T_x$ and $F_j\to T_x$ agree. As the sum of both maps is surjective, it follows that both maps are surjective individually.
\end{proof}

\begin{remark}\label{rem:Lj}
Let $L\in \Pic(C)$ be a line bundle, and $F$ a purely one-dimensional sheaf on~$C$. Then, for every component $C_j$ of $C$, we have
\[(L\otimes F)_j\cong L_j\otimes F_j\,. 
\]
Furthermore, the sheaves $T_x$ associated to $F$ and $L\otimes F$ are isomorphic. 
\end{remark}

Let $F$ be a purely one-dimensional sheaf on $C$.
We denote by $\mu_j(F)$ the multiplicity of $F$ along the component $C_j$. That means that $\mu_j(F)$ is the length of the stalk of $F$ at the generic point of $C_j$.

For every component $C_j$, there is a natural filtration 
\begin{equation*}
F_j=F_{j0}\supset F_{j1}\supset\ldots\supset F_{jm_j}=0\,,
\end{equation*}
where $F_{jk}$ is the saturation of the subsheaf $\reg_X(-kC_j)\cdot F_j=\cI_{kC_j\hookrightarrow m_jC_j}\cdot F_j$ in $F_j$. 
The factors $\gr_{k}F_j=F_{j(k-1)}/F_{jk}$ are pure and hence locally free sheaves 
on the (reduced) curve $C_j$ of a certain rank $r_{jk}$. By construction, there are 
natural maps 
\begin{equation*}
\gr_kF_j\otimes\reg_{C_j}(-C_j)\to \gr_{k+1}F_j
\end{equation*}
that are generically surjective. In particular, the ranks $r_{jk}$ satisfy 
\begin{equation*}
r_{j1}\geq r_{j2}\geq \ldots\,.
\end{equation*} 

\begin{lemma}\label{lem:chiTx}
Let $x\in C_i\cap C_j$. Then
\begin{equation*}
\chi(T_x)\leq \sum_{k,k'} \min\{r_{ik},r_{jk'}\}
\end{equation*}
where the sum runs over all pairs $(k,k')$ with $1\le k\le m_i$ and $1\le k'\le m_j$. 
\end{lemma}
\begin{proof}
We prove this by induction on the lengths of the filtrations of $F_i$ and $F_j$. Note that we can reformulate the assertion in the following way, which is more useful for our proof; compare \autoref{lem:Fses}: \emph{Let $F_i$ and $F_j$ be purely one-dimensional sheaves on $m_iC_i$ and $m_jC_j$ and consider their natural filtrations and rank sequences as defined above. Let $T_x$ be a sheaf supported in $x$ such that there are surjections $F_i\twoheadrightarrow T_x$ and $F_j\twoheadrightarrow T_x$. Then $\chi(T_x)\leq \sum_{k,k'} \min\{r_{ik},r_{jk'}\}$} 

Assume first that the length
of the filtrations on $F_i$ and $F_j$ equal $1$,
so that $F_i$ and $F_j$ are scheme-theoretically supported on the reduced curves $C_i$
and $C_j$, respectively, and are locally free on these reduced curves.
As $T_x$ is a quotient both of $F_i$ and
$F_j$, it must be of the form ${\mathcal O}_x^{\oplus r}$ with rank
$r\leq\mbox{rk}(F_i)=r_{i1}$ and $r\leq\mbox{rk}(F_j)=r_{j1}$. So we have
$\chi(T_x)\leq\mbox{min}\{r_{i1},r_{j1}\}$ and the assertion
is clear in this case.

Now, we assume that the length
of the filtration on $F_i$ is $\geq 2$. So we are
in the situation that $F_{i1}\not=0$.
Let $T'\subset T_x$ denote the image of $F_{i1}$ in $T_x$,
and let $F_j'$ denote the preimage of $T'$ under the surjection $F_j\twoheadrightarrow T_x$.
We get the following maps
\[
 F_{i1}\twoheadrightarrow T' \twoheadleftarrow F_j'\qquad,\qquad \gr_{1}(F_i)\twoheadrightarrow T_x/T' \twoheadleftarrow F_j
\]
all of which are still surjective by construction. As the inclusion 
$F_j' \hookrightarrow F_j$ 
is generically an isomorphism, the sheaves
$F_j'$ and $F_j$ have the same rank sequences.
Hence, applying the induction hypotheses to the left pair of surjections,  we obtain
\begin{equation}\label{eq:Teq1}
\chi(T')\leq\sum_{k,k',\, k'\geq 2}\mbox{min}\{r_{jk},r_{ik'}\}.
\end{equation}
Furthermore, applying the induction hypotheses to the right pair of surjections, gives
\begin{equation}\label{eq:Teq2}
\chi(T_x/T')\leq\sum_{k}\mbox{min}\{r_{jk},r_{i1}\}.
\end{equation}
As $\chi(T_x)=\chi(T')+ \chi(T_x/T')$, the assertion follows from combining \eqref{eq:Teq1} and \eqref{eq:Teq2}.
\end{proof}

\section{Classification of stable sheaves on an extended ADE curve}\label{sect:ADE}

\subsection{Set-up and Notation}
 
We continue with the assumptions from \autoref{subsect:purediv} on our curve $C$: It is embedded as a divisor into a smooth projective surface, its components are smooth, and at most two components intersect transversally in one point.

Associated to $C$ we have the intersection graph $\Gamma=\bigl(V(\Gamma), E(\Gamma), m(\Gamma)\bigr)=(V,E,m)$. The vertex set $V$ consist of one vertex for each component of $C$. The label $m\colon V\to \IN$ sends the vertex corresponding to a component $C_i$ to its multiplicity $m_i$. There is an edge between the vertices corresponding to $C_i$ and $C_j$ for each intersection point of the two components.

\begin{definition}
The \emph{extended ADE graphs} are the labelled graphs displayed in \autoref{fig: dynkin} in the introduction. 
Given an extended ADE graph $\Gamma$, we decompose its vertex set as $V(\Gamma)=I(\Gamma)\sqcup O(\Gamma)$, where $O(\Gamma)$ is the set of vertices with label 1, and $I(\Gamma)$ is the set of vertices of label greater than 1. Later, we will often omit the $\Gamma$ from the notation and simply write
\begin{equation*}
V:=V(\Gamma)\quad,\quad I:=I(\Gamma)\quad,\quad O:=O(\Gamma).
\end{equation*}
\end{definition}

\begin{definition}\label{def: ADEcurve}
 We call $C$ as above an \emph{extended ADE-curve} if 
 \begin{enumerate}
  \item its labeled intersection graph $\Gamma$ is one of the extended ADE Dynkin graphs,
  \item all curves $C_i$ with $i\in I(\Gamma)$ are rational $(-2)$-curves. 
 \end{enumerate}
\end{definition}

From now on, $C$ will always denote an extended ADE curve.
A priori, the $C_o$ with $o\in O(\Gamma)$ (that means the reduced components of $C$) can be of arbitrary genus. However, the geometry of the ambient surface $X$ can impose strong restrictions on the genera. We will see this later in \autoref{sect: linear systems} in the case that $X$ is a K3 surface.

\begin{remark}\label{rem:Cartanmatrix}
Let $S=S_\Gamma$ be the negative of the extended Cartan matrix of type $\Gamma$. Concretely, this means that $S_{vw}=C_v\cdot C_w$ for $(v,w)\in V\times V$ with $v\neq w$, and $S_{v,v}=-2$ for every $v\in V$. In other words, $S$ is the intersection matrix of $C$ in the case that all $C_v$ (also for $v\in O$) are $(-2)$-curves.  

For us, the main property of the matrix $S$ is the following: The integral lattice $(\IZ^{V},
S)$ is even and negative semi-definite and its kernel is generated by 
the primitive vector $m=(m_v)_{v\in V}$ whose entries are the labels of the vertices of $\Gamma$; see e.g.\ \cite[Lem.\ I.2.12]{BHPV}.

Another key property of extended ADE diagrams that we will use multiple times in the text is that twice the label of a vertex is the sum of the labels of adjacent vertices:
\begin{equation}\label{eq:ADElabels}
2m_v=\sum_{w\in V\setminus\{v\}\,,\, C_v\cap C_w\neq \emptyset} m_w\,. 
\end{equation}
\end{remark}

 If all but one component $C_{o}$ are rational, then $E:=C-C_{o}$ is the fundamental cycle of an ADE singularity.

\begin{remark}\label{rem:elliptic}
Extended ADE curves where all components (including the reduced ones) are rational $(-2)$-curves occur as singular fibers of elliptic fibrations.
\end{remark}

\begin{example}
 Let $\pi \colon X \to \bar X$ be the minimal resolution of an ADE singularity $p\in \bar X$. If $p\in \bar C_0\subset \bar X$ is a general curve which is Cartier near $p$, then $\pi^*C_0$ is an extended ADE-curve. 
\end{example}

\begin{lemma}\label{lem:2con}
 Every extended ADE curve $C$ is numerically $2$-connected. This means that, for every decomposition of effective divisors $C=A+B$, we have $A.B\ge 2$.
\end{lemma}

\begin{proof}
 Since the self-intersections of reduced components of $C$ do not affect the intersection number $A.B$, we can pretend that all components (including the reduced ones) are rational $(-2)$-curves. In this case $C^2=0$; see \autoref{rem:Cartanmatrix} or \autoref{rem:elliptic}. By \autoref{rem:Cartanmatrix}, we also have $A^2,B^2\le -2$. Hence, 
 $0=C^2=(A+B)^2=A^2+2A.B+B^2$ implies $A.B\ge 2$.
\end{proof}

\subsection{Sheaves on multiple components}
In this subsection, we study sheaves on the non-reduced components of $C$. Let $B$ be a rational $(-2)$-curve on a smooth surface, and fix some $m\in \IN$. Let $G$ be purely one-dimensional sheaf on $mB$. As above in \autoref{subsect:purediv}, we consider the natural filtration of $G$. More concretely, 
$G_k\subset G$ denotes the saturation of the subsheaf $\reg_X(-kB)\cdot G$ for $0\le k\le m$, and  $\gr_kG=G_{k-1}/G_k$. As $G/G_k=\pure(G_{\mid kB})$, we can also describe the $\gr_k G$ as the kernels
\begin{equation}\label{eq:gradedpieceaskernel}
 0\to \gr_kG\to \pure(G_{\mid kB})\to \pure(G_{\mid(k-1)B})\to 0\,.
\end{equation}
 
\begin{lemma}\label{lem:Ograded}
Let $1\le a\le m$. If $\gr_k G\cong \reg_{B}\bigl(2(k-1)\bigr)$ for all $1\le k\le a$, then we have $G/G_a\cong \reg_{aB}$. 
\end{lemma}
\begin{proof}
 We proof this by induction on $a$. For $a=1$, this is immediate.
For $a>1$, we have the short exact sequence of $\reg_{aB}$-modules
\begin{equation*}
 0\to \gr_a G\to G/G_a\to G/G_{a-1}\to 0\,.
\end{equation*}
Applying the induction hypothesis to the right term, this sequence takes the form 
\begin{equation}\label{eq:Eases}
 0\to \reg_B(2(a-1))\to G/G_a\to \reg_{(a-1)B}\to 0\,.
\end{equation}
Applying $\Hom_{\reg_{aB}}(\_,\reg_B(2(a-1))$ to the structure sequence
\begin{equation}\label{eq:Eases2}
 0\to \reg_B(2(a-1))\to \reg_{aB}\to \reg_{(a-1)B}\to 0
\end{equation}
gives 
\begin{align*}
 0&\to \Hom\bigl(\reg_{(a-1)B},\reg_B(2(a-1))\bigr)\xrightarrow\cong \Hom\bigl(\reg_{aB},\reg_B(2(a-1))\bigr)
 \\&\to \Hom\bigl(\reg_B(2(a-1)),\reg_B(2(a-1))\bigr)\to \Ext^1\bigl(\reg_{(a-1)B},\reg_B(2(a-1))\bigr)
 \\ & 
 \to \Ho^1\bigl(\reg_B(2(a-1))\bigr)=0\,.
\end{align*}
Hence, $\ext^1(\reg_{(a-1)B},\reg_B(2(a-1)))=\hom\bigl(\reg_B(2(a-1)),\reg_B(2(a-1))\bigr)=1$. As \eqref{eq:Eases} is non-splitting (otherwise, we would have $\gr_1G\cong \reg_B^{\oplus 2}$), the two sequences \eqref{eq:Eases} and \eqref{eq:Eases2} agree up to scalar multiplication. In particular, the middle terms agree.
\end{proof}

\begin{lemma}\label{lem:Idualcompspecial}
 The following three conditions are equivalent:
\begin{enumerate}
\item $\gr_k G\cong \begin{cases} \reg_{B}\bigl(2(k-1)\bigr)\quad &\text{for }k=1,\dots, m-1\\
\reg_{B}\bigl(2(m-1)+1\bigr)\quad &\text{for }k=m\end{cases}$,
\item There is a non-splitting short exact sequence
\begin{equation*}
 0\to \reg_{B}\bigl(2(m-1)+1\bigr)\to G\to \reg_{(m-1)B}\to 0\,,
\end{equation*}
\item $G\cong \reg_{mB}(x)$ for some $x\in B$.
\end{enumerate}
\end{lemma}

\begin{proof}
We first note that the implication (i)$\implies$(ii) follows from \autoref{lem:Ograded} (with $a=m-1$). Also note that the equivalence of the three conditions is obvious for $m=1$, where $\reg_{(m-1)B}$ has to be interpreted as the zero sheaf. Thus, from now on, we assume $m\ge 2$.

For (ii)$\implies$ (iii) we first note that, by a computation analogous to the one in the proof of \autoref{lem:Ograded}, applying  
$\Hom_{\reg_{mB}}(\_,\reg_B(2(m-1)+1)$ to the structure sequence
\[
 0\to \reg_B(2(m-1))\to \reg_{mB}\to \reg_{(m-1)B}\to 0
\]
induces an isomorphism 
\[\Hom\bigr(\reg_B(2(m-1)),\reg_B(2(m-1)+1)\bigl)\cong \Ext^1\bigl(\reg_{(m-1)B},\reg_B(2(m-1)+1)\bigr)\,.\] 
This implies that every sequence as in (ii) admits a morphism of short exact sequences
\[
\begin{tikzcd}
0 \rar& \reg_{B}\bigl(2(m-1)\bigr) \dar\rar & \reg_{mB} \dar\rar & \reg_{(m-1)B} \dar\rar &0\\
0 \rar& \reg_{B}\bigl(2(m-1)+1\bigr) \rar & G  \rar & \reg_{(m-1)B} \rar  &0
\end{tikzcd}.
\]
The cokernel of the left vertical arrow is $\reg_x$ for some $x\in X$. Hence, the Snake Lemma gives a short exact sequence $0\to \reg_{mB}\to  G\to \reg_x \to 0$. By \autoref{lem:Oxseq}, this gives $G\cong \reg_{mB}(x)$. 

Let us now show (iii)$\implies$(i). In general, if $F\subset G$ is a purely one-dimensional subsheaf, we have that $\gr_kF\subset \gr_kG$ for all $k$. We can see this from \eqref{eq:gradedpieceaskernel}, as $\pure(F_{\mid kB})\subset \pure(G_{\mid kB})$ (note that a possible kernel of $F_{\mid kB}\to G_{\mid kB}$ zero-dimensional). We apply this to $F=\reg_{mB}\subset G=\reg_{mB}(x)$.
By \autoref{lem:Ograded}, we already know that $\gr_k\reg_{mB}\cong \reg_B(2(k-1))$. The $\gr_k\reg_{mB}(x)$ must be line bundles containing
$\gr_k\reg_{mB}$, hence $\gr_k\reg_{mB}(x)\cong \reg_B(2(k-1)+\delta_k)$ for some $\delta_k\ge 0$. By \autoref{lem:Oxseq}, we have $\chi(\reg_{mB}(x))=\chi(\reg_{mB})+1$. Hence, we must have $\delta_k=1$ for exactly one $k\in \{1,\dots, m\}$ and $\delta_{k'}=0$ for all $k'\neq k$. Since $\reg_B(-B)\cong \reg(2)$, we have non-zero multiplication maps
\[
\reg(2)\otimes \gr_k\reg_{mB}(x)\to \gr_{k+1}\reg_{mB}(x)
\]
which shows that the degrees of the graded pieces have to grow by at least 2 in every step. In other words, the series $\delta_1,\delta_2,\dots, \delta_m$ is non-decreasing. It follows that the unique $k$ with $\delta_k=1$ must be $k=m$.
\end{proof}

\subsection{Line bundles on extended ADE curves}

\begin{lemma}\label{lem:gC}
 We have $g(C):=\ho^1(\reg_C)=1+\sum_{o\in O} g(C_o)$. 
\end{lemma}

\begin{proof}
Note that, for $i\in I$, the sheaf $\reg_{m_iC_i}$ has a filtration whose factors are the line bundles \[\reg_{C_i}, \reg_{C_i}(2),\dots, \reg_{C_i}\bigl(2(m-1)\bigr)\,;\] see \autoref{lem:Ograded}.
This gives $\ho^1(\reg_{m_iC_i})=0$ and 
\[
\ho^0(\reg_{m_iC_i})=\sum_{k=0}^{m_i-1}\ho^0\bigl(\reg(2k)\bigr)=\sum_{k=0}^{m_i-1}(2k+1)=m_i^2\,. 
\]
Also, $\ho^0(\reg_{C_o})=1$ for every $o\in O$. Hence, 
\begin{equation}\label{eq:h0Cv}\ho^0(\reg_{m_vC_v})=m_v^2 \quad\text{for all}\quad v\in V\,.
\end{equation}
We have a short exact sequence
\[
 0\to \reg_C\to \bigoplus_{v\in V}\reg_{m_vC_v}\to \bigoplus_x \reg_{\zeta_x}\to 0,
\]
where the direct sum on the right runs through all intersection points of components of $C$ (with $\zeta_x=(m_vC_v)\cap (m_wC_w)$ denoting the scheme-theoretic intersection when $x=C_v\cap C_w$), and the components of the maps are just the restriction maps; compare \autoref{lem:Fses}.   
Applying the global sections functor to this sequence gives
\[
 0\to \Ho^0(\reg_C)\to \bigoplus_{v\in V}\Ho^0(\reg_{m_vC_v})\to \bigoplus_x \Ho^0(\reg_{\zeta_x})\to \Ho^1(\reg_C)\to  \bigoplus_{o\in O}\Ho^1(\reg_{C_o})\to 0\,.  
\]
As $C$ is a numerically 1-connected Gorenstein curve, we have $\ho^0(\reg_C)=1$; see \cite[Thm.~3.3]{CFHR99}. By 
\eqref{eq:h0Cv}, the second term of the long exact sequence has dimension $\sum_{v\in V}m_v^2$. For $x=C_v\cap C_w$, we have $\ho^0(\reg_{\zeta_x})=m_vm_w$; compare \eqref{eq:zeta}. Hence, it follows from \eqref{eq:ADElabels} that the third term of the long exact sequence has the same dimension $\sum_{v\in V}m_v^2$ as the second term. Taking this information together, we see that the cokernel of the second map is one-dimensional. Hence, we have a short exact sequence
\[
0\to \IC \to \Ho^1(\reg_C)\to  \bigoplus_{o\in O}\Ho^1(\reg_{C_o})\to 0 \,. \qedhere
\]
\end{proof}
By \autoref{prop:Picgeneral}, we get
\begin{cor}\label{cor:Pic}
There is an exact sequence of algebraic groups
\[
1\to G\to \Pic^0(C)\to \prod_{o\in O} \Pic^0(C_o)\to 1,
\]
where $G$ is a one-dimensional smooth connected algebraic group.
\end{cor}

We will see later that $G=\IG_m$ if $\Gamma=A_n$, and $G=\IG_a$ if $\Gamma=D_n, E_6,E_7,E_8$. 

\subsection{Definition of Stability}

The version of stability that we use is \textit{Gieseker-stability} which uses the reduced Hilbert polynomial. Let us quickly review the relevant notions; see \cite{HuybrechtsLehn2010} for a general reference.

Let $X$ be a projective noetherian scheme, and fix a polarisation $H$ on $X$, i.e., an ample line bundle. The Hilbert polynomial $P(F)$ is given by
\begin{equation*}
P(F,n)=\chi(F(nH)).
\end{equation*}
Moreover, the Hilbert polynomial can be uniquely written
in the form
$$P(F,n)=\sum_{i=0}^{\dim(F)}\alpha_i(F)\frac{n^i}{i!},$$
where $\alpha_i(F)$ are rational coefficients for $i=0,\ldots,d=\dim(F)$. The \emph{reduced Hilbert polynomial} is the associated monic polynomial
\begin{equation*}
p(F,n):=\frac{P(F,n)}{\alpha_d(F)}.
\end{equation*}
For two polynomials $p$ and $q$, we write $p>q$ if $p(n)>q(n)$ for all $n\gg 0$.

\begin{definition}
A coherent sheaf $F$ of dimension $d$ is \emph{semi-stable} if $F$ is pure and for any proper purely $d$-dimensional quotient $F\twoheadrightarrow Q$ we have $p(F)\le p(Q)$. It is called \emph{stable}
if $F$ is semi-stable and the inequality is strict, i.e.\ $p(F)<p(Q)$ for
any proper purely $d$-dimensional quotient $F\twoheadrightarrow Q$.
\end{definition}

If $X=C$ is an extended ADE curve, then the reduced Hilbert polynomial is a monic polynomial of degree $1$, hence determined by its constant term. More concretely, by Riemann-Roch,
we have
\begin{equation}\label{eq:HP}
P_{H}(F,n)=\bigl(\sum_v \mu_v(F) e_v\bigr)n +\chi(F)\,.
\end{equation}
for every purely one-dimensional sheaf $F$ on $C$, and hence
\[
 p(F,n)= n +\frac{\chi(F)}{\sum_{v\in V} \mu_v(F) e_v}\,,
\]
where $\mu_v(F)$ denotes the multiplicity of $F$ along $m_vC_v$ and $e_v=\deg(H_{\mid C_v})$.

\subsection{Condition on the polarisation}\label{subsect:polar}

From now on, let us fix some positive integer $\chi>0$. We want to classify certain stable sheaves on $C$ with Euler characteristic $\chi$. In order to talk about stability, we need to fix some polarisation $H$ on $C$. The following assumption is necessary for our classification of stable sheaves of type $m$ and Euler characteristic $\chi$.

\begin{assumption}\label{ass:H}
Let $e_v:=\deg(H_{\mid C_v})$ and $e:=\deg(H)=\sum_{v\in V} m_ve_v$.
Setting $b_v:=\lceil \frac {e_v}e\chi \rceil$ for $v\in V$, we make the following assumption on our polarisation:
\begin{equation}\label{eq:Hassumption}
 \sum_{o\in O}b_o=\chi+|O|-1\,.
\end{equation}
\end{assumption}

This means that $\sum_{o\in O}b_o$ takes the maximal possible value. Indeed, by definition of the $b_o$ as the ceiling of $\frac {e_o}e \chi$, we have
\begin{equation}\label{eq:polar1}
\sum_{o\in O}b_o<|O|+\sum_{o\in O}\frac {e_o}e\chi\,.
\end{equation}
Furthermore, setting $e_I:=\sum_{i\in I}m_ie_i$, for any given polarisation, we have
\begin{equation}\label{eq:polar2}
\sum_{o\in O}\frac {e_o}e\chi\le\frac{e_I+\sum_{o\in O}e_o}e \chi=\frac ee\chi=\chi
\end{equation}
with an equality for $\Gamma=\tilde A_n$, where $I=\emptyset$ hence $e_I=0$, and a proper inequality otherwise. Combining \eqref{eq:polar1} and \eqref{eq:polar2}, we see that for every polarisation, we have $\le$ in
\eqref{eq:Hassumption}.

Under \autoref{ass:H}, we have $\frac{e_o}e\chi\notin \IN$ for $o\in O$. Indeed, otherwise the value of the left-hand of \eqref{eq:polar1} would drop. The condition $\frac{e_o}e\chi\notin \IN$ is necessary to prevent properly semi-stable sheaves.

Looking at \eqref{eq:polar2}, we also see that the $e_i$ need to be small relative to the $e_o$ for our assumption to hold.

The following lemma will be needed in the proof of our classification of stable sheaves in \autoref{prop:chi>0class}.

\begin{lemma}\label{lem:ei} Let $H$ be a polarisation satisfying \autoref{ass:H}
\begin{enumerate}
 \item For every $i\in I$, we have $b_i=1$.
\item
Let $J\subset O$, and $0\le f\le e_I$, and assume that at least one of $J\subsetneq O$ or $f< e_I$ holds.
Then 
\begin{equation*}
1-|J|+ \sum_{j\in J}b_j>\frac{f+\sum_{j\in J} e_j}{e}\chi\,. 
\end{equation*}
\end{enumerate}
\end{lemma}

\begin{proof}
Looking at \eqref{eq:polar1} and \eqref{eq:polar2}, we see that, for $\sum_{o\in O}b_o$ to have a chance to take the maximal value $\chi +|O|-1$, we must have $\sum_{o\in O}\frac {e_o}e\chi>\chi-1$. Hence, $\frac{e_I}e\chi<1$ which gives (i).

Regarding (ii), let us assume for a contradiction that
\begin{equation*}
1-|J|+ \sum_{j\in J}b_j\le\frac{f+\sum_{j\in J} e_j}{e}\chi\,. 
\end{equation*}
Noting that $b_j-1<\frac{e_j}e\chi$ for every $j\in O$, we can add $\sum_{o\in O\setminus J} \frac{e_j}e\chi$ to both sides to get
\begin{align*}
1-|O|+\sum_{o\in O} b_o= 1-|J|+ \sum_{j\in J}b_j +\sum_{o\in O\setminus J}(b_j-1)&\le\frac{f+\sum_{j\in J} e_j}{e}\chi+ \sum_{o\in O\setminus J}\frac{e_j}e\chi\\&= \frac{f+\sum_{o\in O} e_j}{e}\chi\\&\le  \frac{e_I+\sum_{o\in O} e_j}{e}\chi\\&=\chi\,. 
\end{align*}
Note that, by \autoref{ass:H}, not only the right-hand, but also the left-hand side of this chain of inequalities is $\chi$. If $J\subsetneq O$, the first inequality is proper. If $f<e_I$, the second inequality is proper. Either way, we arrive at the contradiction $\chi<\chi$.
\end{proof}

\begin{remark}
If $\chi=1$, we have $b_v=1$ for every $v\in V$. Hence, \autoref{ass:H} is automatically satisfied. Hence, after proving our classification \autoref{thm: easy classification} under \autoref{ass:H}, it follows in particular that stability of sheaves of type $m$ and Euler characteristic $\chi=1$ does not depend on the chosen polarisation.
For higher $\chi$, however, the stability condition depends on the polarisation; see \autoref{subsect:LM} for some further comments on this.
\end{remark}

Note that, given any $\chi>0$, we can find a polarisation $H$ on $C$
which satisfies \autoref{ass:H}. The reason is that we can choose the $e_v$ arbitrarily as long as they are sufficiently large.

A more subtle question is whether we can find a polarisation $H$ on the ambient surface $X$ which restricts to a polarisation of $C$ satisfying \autoref{ass:H}. This is a relevant question as, later in \autoref{subsect:family}, we want to regard the moduli space $\cM$ of sheaves of type $m$ and Euler characteristic $\chi$ as the fibre of $\cN\to |C|$ over $C$, where $\cN$ is the moduli space of $H$-stable sheaves on the surface $X$ with rank $0$, Euler characteristic $\chi$, and first Chern class $c_1=[C]$. In an important class of examples, this is always possible:

\begin{prop}\label{prop:Hone}
 Let $C$ be an extended ADE curve such that all but one component $C_o$ are rational $(-2)$-curves. Then, for every $\chi>0$ there exists a polarisation $H$ of the surface $X$ which satisfies \autoref{ass:H}.
\end{prop}

\begin{proof}
As $E:=C-C_o$ is a configuration of $(-2)$-curves, we can contract it. Let $\pi\colon X\to \bar X$ be the contraction. For $\bar H$ an ample divisor on $\bar X$ and $\eps>0$ sufficiently small, the divisor $H:=\pi^*\bar H-\eps E$ is again ample. It is a priori only a $\IQ$-divisor, but, as \autoref{ass:H} is stable under replacing $H$ by some multiple $mH$, this is not a problem.
By making $\eps$ even smaller, if necessary, we can achieve that $\frac{e_o}e\chi>\chi-1$. Then $b_o=\chi$ and $b_v=1$ for all $v\neq o$, which gives \eqref{eq:Hassumption}.
\end{proof}

However, there are examples of extended ADE curves $C\hookrightarrow X$ and numbers $\chi>0$, where it is impossible to find a polarisation on the surface $X$ satisfying \autoref{ass:H}:

\begin{example} \label{remark: globalisation of polarisation impossible}
Let $X = \IP^2$ and $C = Q+L$ the union of a smooth conic and a general line, which is an extended ADE curve with $\Gamma=\tilde A_1$. For $H = \ko_{\IP^2}(d)$, \autoref{ass:H} becomes
 \[ \left\lceil \frac{d}{3d} \chi\right\rceil + \left\lceil\frac{2d}{3d} \chi\right\rceil = \chi +1
 \]
 which is satisfied (for any positive $d$) if and only if $\chi$ is not divisible by $3$.
 We will see a similar example on a K3 surface later in \autoref{exam: no global polarisation}.
 \end{example}

At least, given a curve $C\hookrightarrow X$ there are always infinitely many values of $\chi>0$ for which we can find a polarisation $H$ on the surface $X$ satisfying \autoref{ass:H}. To see this, note that, as the $C_i$ are rational $(-2)$-curves, we can always contract $E=\sum_{i\in I} m_iC_i$, say via $\pi\colon X\to \bar X$. Then, choose any polarisation $\check H$ on $\bar X$, and set $H_0:=\pi^*\check H$.

This $H_0$ is not ample, hence not a polarisation. Anyway, let us note that, numerically, it satisfies \autoref{ass:H} for every $\chi$ with $\chi\equiv 1\,\,\,\mod \deg(H_0)$. Let us further note that \autoref{ass:H} is stable under small perturbations of $H$ as an $\IQ$-divisor, due to the fact, noted above, that \autoref{ass:H} implies $\frac{e_o}e\chi\notin \IN$. Hence, given some $\chi$ with $\chi\equiv 1\,\,\,\mod \deg(H_0)$, we can find a sufficiently small $\eps>0$ such that $H:=H_0-\eps E$ is a polarisation on $X$ which still satisfies \autoref{ass:H}.

\subsection{Classification of stable sheaves by their components}\label{subsect:chi>0}

\begin{definition}
A coherent sheaf $F$ on $C$ is \emph{of type $m$}, if it is purely one-dimensional, and $\mu_v(F)=m_v$ for every $v\in V$.  
\end{definition}

Every purely one-dimensional sheaf $F$ which is generically a line bundle is a sheaf of type $m$. But, if $C$ is non-reduced, i.e.\ if $\Gamma\in \{\tilde D_n, \tilde E_6 , \tilde E_7, \tilde E_8\}$, the converse is not true. For example, consider $F=\bigoplus_{v\in V} \reg_{C_v}^{\oplus m_i}$. This $F$ has scheme-theoretic support only on $C_{\red}$, but still satisfies $\mu_v(F)=m_v$ for all $v\in V$.

However, our classification of the stable sheaves of type $m$ with respect to a polarisation satisfying \autoref{ass:H} which follows below shows in particular that every \emph{stable} sheaf of type $m$ is generically a line bundle.

Looking at \eqref{eq:HP}, we see that the Hilbert polynomial of a sheaf $F$ of type $m$ is
\[
 P_H(F,n)= en+\chi(F)\,.
\]

Recall that by \autoref{lem:Fses}, for every sheaf $F$ of type $m$, we have an exact sequence
\begin{equation}\label{eq:exactsequenceofsheavesrelatingFandFj2}
0\lra F\lra \bigoplus_{v\in V} F_v\lra \bigoplus_x T_x\lra 0,
\end{equation}
where $F_v=\pure(F_{\mid m_vC_v})$, and $x$ runs through all intersection points  of components of $C$. Also recall that, for $x=C_u\cap C_w$, we write $\zeta_x:=(m_uC_u)\cap (m_wC_w)$, and that, by our transversality assumption, $\reg_{\zeta_x}\cong \IC[x,y]/(x^{m_u},y^{m_w})$. Hence, its socle $(x^{m_u-1} y^{m_w-1})$ is one-dimensional. Let $\zeta_x'\subset \zeta_x$ be the subscheme with $\reg_{\zeta_x'}=\reg_{\zeta_x}/\mathsf{socle}$.  
\begin{definition} Let $F$ be a sheaf of type $m$. 
\begin{enumerate}
 \item For $i\in I$, the sheaf $F_i$ on $m_iC_i$ is 
 \begin{enumerate}
\item \textit{ordinary} if $F_i\cong \reg_{m_iC_i}$.
\item \textit{special} if it satisfies the equivalent conditions of \autoref{lem:Idualcompspecial}, which means that $F_i\cong \reg_{m_iC_i}(x)$ for some $x\in C_i$.
 \end{enumerate}
 \item For $o\in O$, the sheaf $F_o$ on $C_o$ is 
 \begin{enumerate}
\item \textit{ordinary} if $F_o$ is a line bundle of Euler characteristic $b_o:=\lceil \frac{e_j}e \chi\rceil$, compare \autoref{subsect:polar}.
\item \textit{special} if $F_o$ is a line bundle of Euler characteristic $b_o+1$.
 \end{enumerate}
\item Let $x$ be an intersection point of two components of $C$, the sheaf $T_x$ is
\begin{enumerate}
\item \textit{ordinary} if $T_x\cong \reg_{\zeta_x}$.
\item \textit{special} if $T_x\cong \reg_{\zeta'_x}$ where $\reg_{\zeta'_x}=\reg_{\zeta_x}/\mathsf{socle}$.
\end{enumerate}
\end{enumerate}
\end{definition}

Note that there are many sheaves of type $m$ such that the $F_v$ and $T_x$ are neither special nor ordinary. 

\begin{prop}\label{prop:chi>0class}
The following three statements are equivalent for $F$ a sheaf of type $m$:
\begin{enumerate}
 \item $F$ is semi-stable,
 \item $F$ is stable,
 \item exactly one of the sheaves $F_v$ and $T_x$ occurring in \eqref{eq:exactsequenceofsheavesrelatingFandFj2}  is special, and all the others are ordinary.
\end{enumerate}
\end{prop}

\begin{proof}
We start with the (i)$\implies$(iii) part.
For the following notation related to the natural filtration of $F_i$, see \autoref{subsect:purediv}.
For $i\in I$, all the vector bundles $\gr_kF_i$ on $C_i\cong \IP^1$
decompose into line bundles, say
\begin{equation*}
\gr_kF_i=\reg_{C_i}(a_{k1})\oplus\ldots\oplus \reg_{C_i}(a_{kr_{ik}}).
\end{equation*}
Every line bundle $L$ appearing in the decomposition of $\gr_1F_i$ is a quotient of $F$. 
As \begin{equation*}
P(L,n)=e_in+\chi(L), \quad P(F,n)=en+\chi\,,
\end{equation*} 
semi-stability implies $\chi(L)\ge \frac{e_i}e\chi$.  
Applying \autoref{lem:ei}(i) gives $\chi(L)\ge b_i= 1$.
Since $\reg_{C_i}(-C_i)\cong \reg_{C_i}(2)$, and the natural map $\gr_kF_i\otimes\reg_{C_i}(-C_i)\to \gr_{k+1}F_i$ is generically surjective,    
\[\min\{a_{k\alpha}\mid \alpha=1,\dots r_{ik}\}+2\leq \min\{a_{(k+1)\alpha}\mid \alpha=1,\dots,r_{i(k+1)}\}\,.\] 
It follows that $\chi(L)\geq 2k-1$ for every line bundle 
$L$ appearing as a direct summand of $\gr_kF_i$. This yields $\chi(\gr_kF_i)
\geq (2k-1) r_{ik}$. 

For $r\in \IN$, we set $n_{ir}:=\max\{k\;|\; r_{ik}\geq r\}$. Note that the non-increasing sequences $(r_{ik})_k$ and $(n_{ik})_k$ are dual partitions (i.e.\ they have transposed Young diagrams) of $\mu_i(F)=m_i$. In particular,  $\sum_r n_{ik}=m_i=\sum_k r_{ik}$. We get
\begin{equation}\label{eq:SecondInequality}
\chi(F_i)=\sum_k\chi(\gr_kF_i)\ge\sum_k r_{ik}(2k-1)=\sum_r n_{ir}^2
\end{equation}
where the last equality is due to the general identity $\sum_{k=1}^n(2k-1)=n^2$.
Similarly, as $F_o$ is a quotient of $F$ for every $o\in O$, semi-stability gives $\chi(F_o)\ge b_o$.
Taking these inequalities for varying $o\in O$ together, and using \autoref{ass:H}, we get
\begin{equation}\label{eq:Oineq}
 \chi-\sum_{o\in O}\chi(F_o)\le 1-|O|\,.
\end{equation}
Let $x=C_v\cap C_w$ be an intersection point of two components. By a general combinatorical consideration, for example using induction on the lengths of the partitions $(r_{vk})_k$ and $(r_{wk'})_{k'}$, we see that 
\[
 \sum_{k,k'}\min\{r_{vk}, r_{wk'} \}=\sum_{r} n_{wr} n_{vr}\,.
\]
Hence, \autoref{lem:chiTx} becomes 
\begin{equation}\label{eq:Txineq}
 \chi(T_x)\le\sum_r n_{wr} n_{vr}\,.
\end{equation}
Note that, for $o\in O$, we have $n_{o1}=1$ and $n_{or}=0$ for $r\ge 2$. Let $S=S_\Gamma$ be minus the extended Cartan matrix; compare \autoref{rem:Cartanmatrix}.
From the exact sequence \eqref{eq:exactsequenceofsheavesrelatingFandFj2} and the inequalities \eqref{eq:SecondInequality} and \eqref{eq:Oineq}, and \eqref{eq:Txineq}, we conclude that
\begin{eqnarray*}
0=\chi(F)-\sum_{v\in V} \chi(F_v)+\sum_x \chi(T_x)&\le& 1-|O|-\sum_{i\in I}\sum_r n_{ir}^2+\sum_{\{u,w\}\subset V,\, u\neq w}S_{uw} \sum_r n_{ur}n_{wr}\\
&\le& 1-\sum_{v\in V}\sum_r n_{vr}^2+\sum_{\{u,w\}\subset V,\, u\neq w}S_{uw} \sum_r n_{ur}n_{wr}
\\
&=&1+\frac12 \sum_r \sum_{(v,w)\in V\times V} n_{vr} S_{vw} n_{wr}\\&=& 1+\frac12\sum_{r} \vec n_r^t S\vec n_r,
\end{eqnarray*}
where $\vec n_r=(n_{vr})_{v\in V}\in \IZ^{V}$. The properties of $S$ mentioned in \autoref{rem:Cartanmatrix} give that  any non-trivial vector $0\neq v\in \IZ^V$ is either a multiple of $\vec m=(m_v)_{v\in V}$ or satisfies $v^t Sv\le -2$. Since $\vec n_1+\vec n_2+\ldots=\vec m$, the inequality 
$-2\leq \sum_{r} \vec n_r^t S \vec n_r$ shows that we must have $\vec n_1=\vec m$ and $\vec n_r=0$ 
for all $r>1$. We can infer two facts: every subfactor $\gr_kF_i$, with $i\in I$, is a line 
bundle on $C_i$, and, second, as our inequality now reads $0\leq 1+\frac12 
\vec n_1^t S\vec n_1=1$, that exactly {\sl one} of the inequalities used in the 
deduction, namely
\begin{align}\label{eq:ineq1}\chi(\gr_{k}F_i)\ge 2k-1\qquad &\text{for $i\in I$, $k\le m_i$},
\\ \label{eq:ineq2} \chi(F_o)\ge b_o\qquad&\text{for $o\in O$,} \\  \chi(T_x)\leq m_um_w \qquad&\text{for $x=C_u\cap C_w$} \label{eq:ineq3}
\end{align}
fails to be an equality and by an excess of exactly 1.

If $\chi(\gr_{k}F_i)> 2k-1$ for some $k\le m_i$, we must already have $k=m_i$. Otherwise, as the Euler characteristic of the factors grows by at least 2 in every step, we would also have $\chi(\gr_{k+1}F_i)> 2(k+1)-1$ as a second proper inequality. Hence, in this case, \autoref{lem:Ograded} and \autoref{lem:Idualcompspecial} show that $F_i$ is special, and all other $F_v$ and all $T_x$ are ordinary. 

If \eqref{eq:ineq2} is a proper inequality, then $F_o$ is special, and all other $F_v$ and all $T_x$ are ordinary; see again \autoref{lem:Ograded}. 

Let $\chi(T_x)= m_um_w-1$ for $x=C_u\cap C_w$. Then the other two inequalities, \eqref{eq:ineq1} and \eqref{eq:ineq2},  must be equalities, which implies that all $F_v$ are ordinary, hence line bundles. As the maps $F_u\to T_x$ and $F_w\to T_x$ are surjections, we must have that $T_x$ is a quotient of $(F_u)_x\cong \reg_{m_uC_u,x}$ and of $(F_w)_x\cong \reg_{m_wC_w,x}$, hence of $\reg_{\zeta_x}$. The only quotient of $\reg_{\zeta_x}$ of dimension $m_um_w-1$ is $\reg_{\zeta_x'}$, which means that $T_x$ is special.       

We now come to the proof of the implication (iii)$\implies$(ii).
All purely one-dimensional quotients $F''$ of $F$ are of the form $F''=\pure(F_{|Z})$ where $Z=\sum_{v\in V} m_v''C_v$ for some $0\le m_v''\le m_v$; see \autoref{lem:quotientsarerestrictions}.
Hence, for any such proper quotient $F\twoheadrightarrow F''=\pure(F_{\mid Z})$, we need to show the inequality $p(F,n)<p(F'',n)$. 
We consider the short exact sequence 
\begin{equation}\label{eq:F''ses}
0\lra F''\lra \bigoplus_{v\in V} F''_v\lra \bigoplus_x T''_x\lra 0
\end{equation}
which is the analogue of \eqref{eq:exactsequenceofsheavesrelatingFandFj2} for $F''$; see \autoref{lem:Fses}. We set
  \begin{equation*}
   J:=\bigl\{j\in O\mid m_j''=1\bigr\}=\bigl\{j\in O\mid C_j\subset \supp(F'')\bigr\}=\bigl\{j\in O\mid F''_j\neq 0\bigr\}\,.
   \end{equation*}
For $j\in J$, we have $F''_j=F_j$. Hence, 
\begin{equation*}
\chi(F''_j)\ge b_j \qquad \text{ for $j\in J$.}
\end{equation*}
For $i\in I$, we have $\gr_kF_i''\cong \gr_kF_i \cong \reg_{C_i}(2k-2)$ for $1\le k\le m_i''$ with one possible exception: If $k=m_i''=m_i$ and $F_i$ is special, we have $\gr_kF_i'' \cong \reg_{C_i}(2k-1)$. 
Anyway, we get
\begin{equation*}
 \chi(F_i'')\ge \sum_{k=1}^{m_i''} (2k-1)=(m_i'')^2\,.
\end{equation*}
Furthermore, for $x=C_u\cap C_w$, we have $T''_x=\reg_{m_u''C_u\cap m_w''C_w}$, with one possible exception: If $m_u''=m_u$, $m_w''=m_w$ and $T_x$ is special, we have $T''_x=T_x=\reg_{\zeta_x'}$. Anyway,  $T''_x$ is a quotient of $\reg_{m_u''C_u\cap m_w''C_w}$ which gives
\begin{equation*}
 \chi(T''_x)\le m_u''m_w'' \,.
\end{equation*}
In summary, using \eqref{eq:F''ses}, we get
\begin{align*}
\chi(F'')&\ge\sum_{j\in J}b_j+ \sum_{i\in I} (m_i'')^2 -\sum_x m_u''m_w'' \\
&=\sum_{j\in J}b_j-|J|+\sum_{j\in O}(m_j'')^2+ \sum_{i\in I} (m_i'')^2 -\sum_x m_i''m_j'' \\
&= \sum_{j\in J}b_j-|J|- \frac 12 (\vec m'') S\vec m''\\
&\ge \sum_{j\in J}b_j-|J|+1, 
\end{align*}
where the last inequality is due to the fact that $\vec m''=(m''_v)_{v\in V}$ is not a multiple of $\vec m$ and \autoref{rem:Cartanmatrix}. Combining this with \autoref{lem:ei}(ii) gives
\[\chi(F'')>\frac{\sum_{i\in I}m_i''e_i+\sum_{j\in J} e_j}{e}\chi\,.\]
This is exactly what we need as, by \eqref{eq:HP}, 
we have 
\begin{equation*}
 P(F'',n)=\bigl(\sum_{i\in I}m_i''e_i+\sum_{j\in J} e_j\bigr) n+\chi(F'')\,.
\end{equation*}
As the implication (ii)$\implies$(i) is trivial, we finished the proof.
\end{proof}

\subsection{Presentation of stable sheaves as generalised divisors}

\begin{lemma}\label{lem:lb}
 Let $F$ be a purely one-dimensional sheaf on $C$, and let $x=C_v\cap C_w$ be an intersection point of two components. If $T_x\cong \reg_{\zeta_x}$ is ordinary, and $F_v$ and $F_w$ are line bundles in a neighbourhood of $x$, then $F$ is a line bundle in a neighbourhood of $x$.
\end{lemma}

\begin{proof}
Over a sufficiently small open neighbourhood of $x$, the sequence \eqref{eq:exactsequenceofsheavesrelatingFandFj2} becomes
\begin{equation}\label{eq:Flocal}
 0\to F\to \reg_{m_vC_v}\oplus \reg_{m_wC_w}\to \reg_{\zeta_x}\to 0\,.
\end{equation}
As, on the affine neighbourhood, every invertible section of $\reg_{\zeta_x}$ can be lifted to one of $\reg_{m_vC_v}$ and to one of $\reg_{m_wC_w}$, the sequence \eqref{eq:Flocal} is isomorphic to the sequence where both maps $\reg_{m_vC_v}\to \reg_{\zeta_x}$ and $\reg_{m_wC_w}\to \reg_{\zeta_x}$ are just the restriction maps. This means that $F\cong \reg_C$ on our open affine neighbourhood of $x$.  
\end{proof}

\begin{prop}\label{prop:classOx}
 Let $F\in \Coh(X)$ be a stable sheaf of type $m$. 
 \begin{enumerate}
\item
Let $x$ be an intersection point of two components of $C$, and let $L\in \Pic(C)$ with $\chi(L_{\mid C_v})=b_v$ for all $v\in V$. Then \[F=L(x)\] is a stable sheaf with $T_x$ special. Conversely, every stable sheaf of type $m$ with $T_x$ special is of this form. 
 \item Let $i\in I$, $t\in C_i$ a point which is not an intersection point with another component of $C$, and let $L\in \Pic(C)$ with $\chi(L_{\mid C_v})=b_v$ for all $v\in V$. Then   
  \[
   F= L(t)
  \]
  is a stable sheaf with $F_i$ special. Conversely, every stable sheaf of type $m$ with $F_i$ special is of this form. 
\item Let $o\in O$. Every $F=M\in \Pic(C)$ with $\chi(M_{\mid C_v})=b_v$ for all $v\in V\setminus \{o\}$ and $\chi(M_{\mid C_o})=b_o+1$ is stable with $F_o$ special. Conversely, every stable sheaf of type $m$ with $F_o$ special is of this form.  
 \end{enumerate}
\end{prop}

\begin{proof}
Let $x=C_i\cap C_j$ be an intersection point of two components of $C$.
We first claim that the sequence of \autoref{lem:Fses} for $\reg_C(x)$ takes the form 
\begin{equation}\label{eq:OCxses}
 0\to \reg_C(x)\to \bigoplus_{v\in V} \reg_{m_vC_v}\to \reg_{\zeta'_x}\oplus \bigoplus_{y\neq x} \reg_{\zeta_y}\to 0, 
\end{equation}
where all components of the second map are the natural restriction maps. Indeed, let us denote the kernel of 
\eqref{eq:OCxses} by $K$ with the goal to show $K\cong \reg_C(x)$. There is a morphism of short exact sequences
\[
\begin{tikzcd}
0 \rar& \reg_{C} \dar\rar & \bigoplus_{v\in V} \reg_{m_vC_v} \dar{\id}\rar & \reg_{\zeta_x}\oplus \bigoplus_{y\neq x} \reg_{\zeta_y}\to 0  \dar\rar &0\\
0 \rar& K \rar & \bigoplus_{v\in V} \reg_{m_vC_v} \rar & \reg_{\zeta'_x}\oplus \bigoplus_{y\neq x} \reg_{\zeta_y} \rar  &0
\end{tikzcd}.
\]
As the kernel of the right vertical map is $\reg_x$, Snake Lemma gives a short exact sequence
$0\to\reg_C\to K\to \reg_x\to 0$. Hence, by \autoref{lem:Oxseq}, $K\cong \reg_C(x)$. 
Now, by \autoref{rem:Lj}, it follows that $L(x)=L\otimes \reg_C(x)$, satisfies \autoref{prop:chi>0class}(iii) with $T_x$ special. Hence, it is stable.  

Conversely, let $F$ be a stable sheaf with $T_x$ special. Then, by \autoref{prop:chi>0class}, all $F_v$ are ordinary, hence line bundles on $m_vC_v$. Analogously to the proof of \autoref{lem:lb}, we can find an open neighbourhood $x\in U\subset C$ over which  
\eqref{eq:exactsequenceofsheavesrelatingFandFj2} is isomorphic to
\[ 0\to F\to \reg_{m_iC_i}\oplus \reg_{m_jC_j}\xrightarrow{(r,r)} \reg_{\zeta'_x}\to 0,
\]
where $r$ denotes the appropriate restriction maps. Hence, over $U$, we get a morphism of short exact sequences
\[
\begin{tikzcd}
0 \rar& \reg_{C} \dar\rar & \reg_{m_iC_i}\oplus \reg_{m_jC_j} \dar{\id}\rar & \reg_{\zeta_x} \dar{r}\rar &0\\
0 \rar& F \rar & \reg_{m_iC_i}\oplus \reg_{m_jC_j}  \rar & \reg_{\zeta'_x} \rar  &0
\end{tikzcd}.
\]
As the kernel of the restriction map $r\colon \reg_{\zeta_x}\to \reg_{\zeta'_x}$ is $\reg_x$, Snake Lemma gives a short exact sequence
\[
 0\to \reg_U\to F_{\mid U}\to \reg_x\to 0\,.
\]
By \autoref{lem:Oxseq}, this gives $F_{\mid U}\cong \reg_U(x)$. 
Furthermore, by \autoref{prop:chi>0class}, all sheaves $T_{x'}$ for intersection points $x'\neq x$ are ordinary. Hence, by \autoref{lem:lb}, $F_{\mid C\setminus \{x\}}$ is a line bundle. The assertion now follows by \autoref{lem:lbcrit}.

For part (ii), note that $F=L(t)$ satisfies the assumptions of \autoref{prop:chi>0class}(iii) with $F_i$ special. Hence, it is stable. 

Conversely, let $F$ be a stable sheaf with $F_i$ special. Then, by definition, $F_i\cong \reg_{m_iC_i}(t)$ for some $t\in C_i$. By \autoref{lem:End1}(iii), we get that $t$ cannot be an intersection point of $C_i$ with another component. The assertion that $F$ is of the desired form $L(t)$ follows by \autoref{lem:lb} together with \autoref{lem:lbcrit}. 

Part (iii) follows directly from \autoref{lem:lb}.  
\end{proof}

 Let $L$ be a line bundle on $C$ with $\chi(L_{\mid C_v})=b_v$ for all $v\in V$. Then, for $o\in O$, and $x\in C_o$ a point which is not an intersection point with another component of $C$, the line bundle $M=L(x)$ satisfies the properties of \autoref{prop:classOx}(iii). Hence, we can summarise \autoref{prop:classOx} as follows, which is 
 
 \begin{cor}\label{cor:class}
  The stable sheaves of type $m$ are exactly the sheaves of the form $F=L(x)$ for some line bundle $L$ on $C$ with $\chi(L_{\mid C_v})=b_v$ for all $v\in V$ and some $x\in C$.
 \end{cor}

 Note, however, that $F$ does not determine $L$ and $x$, which means that we can have $L(x)\cong L'(x')$ for $L\not\cong L'$ and $x\neq x'$. For singular stable sheaves, however, we will see a result in the direction of uniqueness of the presentation $L(x)$ in the next subsection.

Note that, combining \autoref{cor:class} with the equivalence of the first two points in \autoref{prop:chi>0class}, we have now proved the classification \autoref{thm: easy classification} from the introduction.

\subsection{Uniqueness for singular stable sheaves}

Our goal for this subsection is to prove

\begin{prop}\label{prop:unique}
Let $x,x'\in C_{\mathsf{sing}}$ and $L,L'\in \Pic(C)$. Then
\[
 L(x)\cong L'(x')\quad\iff\quad x=x' \text{ and } L_{\mid C_v}\cong L'_{\mid C_v} \text{ for all } v\in V\,.
\]
\end{prop}

We first need some notation and a few lemmas.
For $v\in V$, we write $C^v:=C-C_v$. By \eqref{eq:ADElabels}, we have
\begin{equation}\label{eq:intnumber}
 C^v.C_v=2\,.
\end{equation}

\begin{lemma}\label{lem:1con}
 The curve $C^v$ is numerically $1$-connected. This means that, for every decomposition of effective divisors $C^v=A+B$, we have $A.B\ge 1$.
\end{lemma}

\begin{proof}
 Let $C^v=A+B$. By \eqref{eq:intnumber}, we must have $A.C_v<2$ or $B.C_v<2$. Let us assume without loss of generality that $A.C_v<2$. We have $C=A+(B+C_v)$. As, by \autoref{lem:2con}, the curve $C$ is numerically $2$-connected, 
 \[
  2\le A.(B+C_v)=A.B+A.C_v\,.
 \]
Our assumption $A.C_v<2$ now gives $1\le A.B$, as asserted.
\end{proof}

\begin{lemma}\label{lem:sesCv}
Let $x\in C_{\mathsf{sing}}\cap C_v$ for some $v\in V$ (if $v\in I$, then $C_{\mathsf{sing}}\cap C_v=C_v$, and if $v\in O$, then $C_{\mathsf{sing}}\cap C_v$ is the set of intersection points of $C_v$ with other components). Then, there is an exact sequence
 \[
  0\to \reg_{C_v}(-C^v +x)\to \reg_C(x)\to \reg_{C^v}\to 0\,.
 \]
Note that $\deg\reg_{C_v}(-C^v+x)=-1$. 
\end{lemma}

\begin{proof}
 We have a chain of inclusions
 \[\cI_{C_v\hookrightarrow C}=\reg_{C^v}(-C_v)\subset \reg_{C}(-x)\subset\reg_{C}\,.\]
Snake lemma gives a commutative diagram with exact columns and rows
\[
\begin{tikzcd}
 & 0\dar & 0\dar & \\
& \reg_{C^v}(-C_v)\dar \rar[equal] & \reg_{C^v}(-C_v)\dar\\
0 \rar& \reg_{C}(-x)\dar\rar & \reg_{C} \dar\rar & \reg_x \dar[equal]\rar &0\\
0 \rar& \reg_{C_v}(-x) \dar\rar & \reg_{C_v} \dar \rar & \reg_x \rar  &0\\
& 0 & 0 &
\end{tikzcd}.
\]
 Taking the dual of the left column gives the desired short exact sequence; see \autoref{lem: decomposition sequence}.
\end{proof}

For $x\in C_{\mathsf{sing}}$, the natural map $\sEnd_{\reg_C}(\cI_x)\to \reg_C(x)$ is an isomorphism, which endows $\reg_C(x)$ with the structure of a commutative $\reg_C$-algebra; see \autoref{lem:End1} and \autoref{lem:End2}. We write
\[
 \nu^x\colon C^x:=\Spec_{\reg_C}\bigl(\reg_C(x)\bigr)\to C\,,
\]
which is an isomorphism away from $x$.
\begin{lemma}\label{lem:nu}
 Let $L,L'\in \Pic(C)$ and $x\in C_\sing$. Then
 \[
  L(x)\cong L'(x)\quad\iff\quad \nu^{x*}L\cong \nu^{x*}L'\,.
 \]
\end{lemma}
\begin{proof}
  By definition, we have $\reg_C(x)\cong \nu^x_*\reg_{C^x}$. Hence, projection formula gives 
\[
 L(x)\cong \nu^x_*\nu^{x*}L  \quad,\quad L'(x)\cong \nu^x_*\nu^{x*}L'\,.
\]
By \autoref{lem:End1}, \autoref{lem:End2}, and \autoref{rem:End}, we have $\sEnd\bigl(\reg_C(x)\bigr)\cong \reg_C(x)$. Hence,
\[
 \Hom_{\reg_C}(L(x),L'(x))\cong \Hom_{\reg_C}(L,L'(x))\cong \Hom_{\reg_C}(L,\nu^x_*\nu^{x*}L')\cong \Hom_{\reg_{C^x}}(\nu^{x*}L,\nu^{x*}L')\,,
\]
which implies the assertion.
\end{proof}

By parts (ii) of \autoref{lem:End1} and \autoref{lem:End2}, we have 
$(C^x)_{\mathsf{red}}\cong C_{\mathsf{red}}$ except for one case: If $\Gamma=\tilde A_n$ and $x$ is an intersection point of two components, then $C_x$ is an $A_{n+1}$-chain of curves, and $\nu^x$ glues the two extremal curves of the chain at $x$; compare \autoref{lem:End2}(iv).
Anyway, $C$ and $C^x$ have the same irreducible components, which get identified by $\nu^x$.

\begin{lemma}\label{lem:PicCx}
 The morphism $\Pic(C^x)\to \prod_{v\in V} \Pic(C_v)$, given by restriction of line bundles to the irreducible components, is an isomorphism.
\end{lemma}

\begin{proof}
As, by \autoref{lem:1con}, $C^v=C-C_v$ is a numerically $1$-connected Gorenstein curve, we have $\ho^0(\reg_{C^v})=1$; see \cite[Thm.~3.3]{CFHR99}. Hence, taking global sections of the short exact sequence of \autoref{lem:sesCv} gives \begin{equation}\label{eq:h0equal}\ho^0(\reg_C(x))=1=\ho^0(\reg_C)\,.\end{equation}
On the other hand, the short exact sequence $0\to \reg_C\to \reg_C(x)\to \reg_x\to 0$ yields \[\chi(\reg_C(x))=\chi(\reg_C)+1\,.\] Combining this with \eqref{eq:h0equal} gives $\ho^1(\reg_C(x))=\ho^0(\reg_C)-1$. Together with \autoref{lem:gC}, this implies
$\ho^1(\reg_C(x))=\sum_{o\in O}\ho^1(\reg_{C_o})$, which by \autoref{cor:Piciso} proves the assertion.
\end{proof}

\begin{proof}[Proof of \autoref{prop:unique}]
As $x\in C_{\sing}$ is the unique point where $L(x)$ fails to be a line bundle, $L(x)\cong L'(x')$ implies $x=x'$; see \autoref{lem:Cxsing}. Now, the assertion follows by combining \autoref{lem:nu} and \autoref{lem:PicCx}.
\end{proof}

\section{Description of the moduli space}\label{sect:modulispacedescr}

Fix some $\chi>0$. Let $\cM$ be the moduli space of stable sheaves of type $m$ and Euler characteristic $\chi$ on $C$. In this section, we describe $\cM_{\red}$. The main goal is to prove \autoref{thm: easy moduli} from the introduction,  respectively its more technical cousin \autoref{thm: full moduli}.

\subsection{An overparametrisation of the stable sheaves}\label{subsect:over}

Let $\wJ\subset \Pic(C)$ denote the connected component of the Picard scheme parametrising line bundles $L\in \Pic(C)$ with $\chi(L_{\mid C_v})=b_v$ for all $v\in V$.  

\autoref{prop:Picgeneral} and \autoref{cor:Pic} provide the following description of $\wJ$.
Restriction of a line bundle to the components of $C$ gives a morphism 
\[
\pi\colon \wJ\to J:=\prod_{o\in O} \Pic_{b_o}(C_o)\quad,\quad L\mapsto (L_{\mid C_o})_{o\in O},
\]
where $\Pic_{b_o}(C_o)$ is the moduli space of line bundles on $C_o$ of Euler characteristic $b_o$. 
Note that, for $L\in \wJ$, we have $L_{\mid C_i}\cong \reg_{C_i}$ as $b_i=1$ and $C_i\cong \IP^1$ for every $i\in I$.
Let $G$ be the one-dimensional connected algebraic group of line bundles on $C$ whose restriction to every component is trivial. Then $G$ acts freely on $\wJ$ by the tensor product, and $\pi$ is the quotient by this action. In other words, $J=\wJ/G$.

\begin{lemma}\label{lem:flat}
Let $\Delta\subset C\times C$ denote the diagonal. Then, the sheaf 
 \[
  \reg_{C\times C}(\Delta):=(\cI_{\Delta\hookrightarrow C\times C})^\vee
 \]
is flat over $C$ (via either of the projections).
\end{lemma}

\begin{proof}
 By \cite[Prop.\ 2.10]{Har--gen}, we have a short exact sequence,
 \[
  0\to \reg_{C\times C}\to \reg_{C\times C}(\Delta) \to \cL_\Delta\to 0,
 \]
where $\cL_\Delta$ is the push-forward of some line bundle on $\Delta$ along the embedding $\Delta\hookrightarrow C\times C$. Since the first and the last term of the sequence are flat over $C$, the same holds for the middle term.
\end{proof}

Let $\cP$ be a universal line bundle on $\wJ\times C$.
On $\wJ\times C\times C$, we consider the coherent sheaf
\[
 \cF:=\pr_{13}^*\cP\bigl(\wJ\times \Delta\bigr):=\pr_{13}^*\cP\otimes \pr_{23}^*\reg_{C\times C}(\Delta)\,.
\]
By \autoref{lem:flat}, $\cF$ is a flat family of sheaves on $C$ over $\wJ\times C$, via $\pr_{12}$. For $(L,x)\in \wJ\times C$, its fibre is $\cF_{(L,x)}\cong L(x)$. By \autoref{cor:class}, this means that the fibres of $\cF$ are exactly the stable sheaves of type $m$. Hence, we get a surjective classifying morphism 
\[
 \phi\colon \wJ\times C\to \cM\quad,\quad (L,x)\mapsto L(x)\,.
\]

\begin{lemma}\label{lem:dphi}
 For every $(L,x)\in\wJ\times C$, the kernel of the differential 
 \[
  d\phi_{(L,x)}\colon T_{\wJ\times C}(L,x)\to T_\cM(L(x))
 \]
is one-dimensional. If $x\in C_{\sing}$, the kernel of $d\phi_{(L,x)}$ consists exactly of the vectors which are tangent to the $G$-orbit of $(L,x)$ (with $G$ acting trivially on the second factor). 
\end{lemma}

\begin{proof}
Note that we have a direct sum decomposition
\begin{equation}\label{eq:dirsum}
T_{\wJ\times C}(L,x)= T_{\wJ}(L)\oplus T_C(x)\,. 
\end{equation}
 Let us start with the case that $x$ is a smooth point of $C$. This means that $x\in C_o$ for some $o\in O$, and $x$ is not an intersection point with another component. Let $\wJ_o$ denote the component of $\Pic(C)$  of line bundles $M$ with $\chi(M_{\mid C_o})=b_o+1$ and $\chi(M_{\mid C_v})=b_v$ for all $v\in V\setminus \{o\}$. These are exactly the stable line bundles with $M_o$ special; see \autoref{prop:classOx}. Hence, $\wJ_o\subset \cM$ is an open subscheme. The restriction $\phi_x\colon \wJ\cong \wJ\times \{x\}\to \cM$ factorises over this open subscheme as the isomorphism
 \[
 \wJ\to \wJ_o\quad,\quad L\mapsto L(x)\,.  
 \]
Hence, the restriction of $d\phi_{(L,x)}$ to the first summand of \eqref{eq:dirsum} is an isomorphism. This means that the kernel of $d\phi_{(L,x)}$ is isomorphic to the second summand $T_C(x)$, which is one-dimensional for the smooth point $x$.

Let now $x\in C_{\sing}$. Let $v=(v_1,v_2)\in T_{\wJ\times C}(L,x)= T_{\wJ}(L)\oplus T_C(x)$. This corresponds to a morphism
\[
 v=(v_1,v_2)\colon \eta\to \wJ\times C,
\]
where $\eta:=\Spec\IC[\eps]/(\eps^2)$.
We assume that $d\phi(v)=0$ which means that 
\begin{equation}\label{eq:trivdef}
 (v\times \id_C)^*\cF\cong (q\times \id_C)^*(L(x)),
\end{equation}
where $q\colon \eta\to \Spec \IC$. From this, we want to deduce that $v\in \ker(d(\pi\times \id_C))$ (with $v$ regarded as a tangent vector). In other words, that $(\pi\times \id_C)\circ v$ (with $v$ regarded as a morphism) equals the constant morphism
\[
\eta\xrightarrow{q}\Spec \IC\xrightarrow{(\pi(L),x)} J\times C \,.
\]
Let $x\in U\subset C$ be an open affine neighbourhood such that 
\[
(v_1\times \id)^*\cP_{\mid \eta\times U}\cong \reg_{\eta\times U}\quad,\quad L_{\mid U}\cong \reg_U\,.
\]
Then $(v\times \id_C)^*\cF_{\mid \eta\times U}\cong (v_2\times \id_U)^*\reg_{U\times U}(\Delta)\cong \reg_{\eta\times U}(\Gamma_{v_2})$ and \[(q\times \id_C)^*L(x)_{\mid \eta \times U}\cong (q\times \id_U)^*\reg(x)\cong \reg_{\eta\times U}(\eta\times x)\,.\]
Hence, by \eqref{eq:trivdef}, we have $\reg_{\eta\times U}(\Gamma_{v_2})\cong \reg_{\eta\times U}(\eta\times x)$. This means that the generalised divisors $\Gamma_{v_2}$ and $\eta\times x$ on $\eta\times U$ are linearly equivalent; see \cite[Prop.\ 2.8(c)]{Har--gen}. By parts (i) of \autoref{lem:End1} and \autoref{lem:End2}, we have that the natural map $\sEnd(\reg_U(-x))\to \reg_U(x)$ is an isomorphism. Hence, the natural map $\sEnd(\reg_{\eta\times U}(-\eta\times x))\to \reg_{\eta\times U}(\eta\times x)$ is an isomorphism too. This implies that $\End\bigl(\reg_{\eta\times U}(\eta\times x)\bigr)$ acts transitively on $\Gamma \bigl(\reg_{\eta\times U}(\eta\times x)\bigr)$. This, in turn, says that the  generalised divisor $\eta \times x$ on $\eta\times U$ is linearly equivalent only to itself; see \cite[Rem.\ 2.9]{Har--genbil}. Hence, we must have $\eta\times x=\Gamma_{v_2}$ as subschemes of $\eta\times U$. This means that $v_2$ equals the constant map 
\[
 v_2\colon \eta\xrightarrow q\Spec\IC\xrightarrow x C\,,
\]
or, in terms of tangent vectors, $v_2=0$. 
In particular, \eqref{eq:trivdef} becomes
\[
(v_1\times \id_C)^*\cP(\eta \times x)\cong (v\times \id_C)^*\cF\cong (q\times \id_C)^*(L(x))\cong ((q\times \id_C)^*L)(\eta\times x)\,.
\]
A straight-forward variant of \autoref{lem:nu} shows that this is equivalent to 
\[
(v_1\times \nu^x)^*\cP\cong (\id_\eta\times \nu^{x})^*(v_1\times \id_C)^*\cP\cong (\id_\eta\times \nu^{x})^*(q\times \id_C)^*L\cong (q\times \nu^x)^*L 
\]
which means that $(v_1\times \nu^x)^*\cP$ is a trivial deformation of $\nu^{x*}L$. Hence, $v_1\in \ker(d\pi')$ where 
\[
 \pi'\colon \wJ\to \Pic_b(C^x)\quad,\quad N\mapsto \nu^{x*}N
\]
and $\Pic_b(C^x)\subset \Pic(C^x)$ denotes the connected component parametrising line bundles whose restriction to every component $C_v$ of $C^x$ is $b_v$ (recall that $C^x$ and $C$ have the same irreducible components).
By \autoref{lem:PicCx}, the morphisms $\pi$ and $\pi'$ agree up to an isomorphism of their target varieties. Hence, also $v_1\in \ker(d\pi)$.
\end{proof}

\subsection{Moduli of singular stable sheaves}

We consider the restriction 
\[
 \phi_{\sing}:=\phi_{\mid \wJ\times (C_\sing)_{\red}}\colon \wJ\times (C_\sing)_{\red}\to \cM 
\]
of the classifying map $\phi$ from the last subsection, and the free action of $G$ on $\wJ\times (C_\sing)_{\red}$, given by the tensor product of line bundles on the first factor, and by the trivial action on the second factor. By \autoref{prop:unique}, $\phi_{\sing}$ is invariant under this action. We denote the factorisation over the quotient $(\pi\times \id)\colon \wJ\times (C_\sing)_{\red}\to J\times (C_\sing)_{\red}$ by $\psi\colon J\times (C_\sing)_{\red}\to \cM$.

\begin{prop}\label{prop:Csingred}
The morphism $\psi\colon J\times (C_\sing)_{\red}\to \cM$ is a closed embedding. 
\end{prop}
\begin{proof}
By \autoref{prop:unique}, $\psi$ is injective, and by \autoref{lem:dphi} its differential in every point is injective.
\end{proof}

As the image of $\psi$ consist exactly of the singular stable sheaves of type $m$, we now have a description of the reduced moduli space of singular sheaves of type $m$. We also see that taking the reduction of the moduli space really makes a difference in the cases $\Gamma\in \{\tilde D_n,\tilde E_6, \tilde E_7,\tilde E_8\}$. If $C_i$ is a non-reduced component of $C$, then, for any $L\in \wJ$ and $x\in C_i$, the tangent space
\[
 T_{\wJ\times C}(L,x)=T_{\wJ}(L)\oplus T_C(x)
\]
is of dimension $\dim \wJ +2=\dim J+3$. Hence, by \autoref{lem:dphi}, the dimension of $T_{\cM}(L(x))$ is at least $\dim J+2=\dim(J\times C_i)+1$, showing that $J\times C_i\subset J\times (C_\sing)_{\red}$ is a multiple component of the moduli space.
We think that the following more precise result should hold.

\begin{conjecture}\label{Conj:multi}
Let $\Gamma\in \{\tilde D_n,\tilde E_6, \tilde E_7,\tilde E_8\}$, which gives $C_\sing=\sum_{i\in I}m_iC_i$. Then there is a closed embedding $J\times C_{\sing}\hookrightarrow \cM$ whose image is the moduli space of singular stable sheaves of type $m$. 
\end{conjecture}

It feels like the proof should be pretty straight-forward considering all the information we already collected. However, there are some technical difficulties which can be avoided by working with reduced schemes as in \autoref{prop:Csingred}, and the authors did not have the drive to work through these difficulties. 

The moduli of stable line bundles of type $m$ is already described by \autoref{prop:classOx} together with  
\autoref{cor:Pic}. It is a disjoint union $\coprod_{o\in O} U_o$ where $U_o$ is the moduli space of line bundles $F$ with $F_o$ special. Every $U_o$ is isomorphic to $\wJ$, hence a $G$-torsor over $J$. 

The question that remains is how exactly the moduli of stable line bundles and the (reduced) moduli of stable singular sheaves are attached to each other to form $\cM_{\red}$. The key observation to answer this question is that, given $o\in O$, the stable sheaves $F$ with either $F_o$ special (in which case $F$ is a line bundle) and the stable sheaves with $T_x$ special for some intersection point of $C_o$ with another component of $C$ (in which case $F$ is singular in $x$) have something in common: They can be written as an extension 
\[
 0\to K\to F\to M\to 0,
\]
where $K\in \Pic(C_o)$ with $\chi(K)=b_o-1$ and $M\in \Pic(C^o)=\Pic(C-C_o)$ with $\chi(M_{\mid C_v})=b_v$ for all $v\in V\setminus\{o\}$. This will allow us to realise the locus in $\cM$ of stable sheaves with either $F_o$ or $T_x$ special as a projectivisation of a relative extension bundle over $J$. We carry out the details in the following subsections.

\subsection{Universal relative extension bundles}\label{subsect:universalext} 
Let $f\colon X\to S$ be a flat morphism and $\cF, \cG$ coherent sheaves on $X$.
For $i\ge 0$, the \emph{relative extension sheaf} is defined by 
\[
\sExt_f^i(\cF,\cG):=\mathcal H^i\bigl(Rf_*R\sHom(\cF, \cG)\bigr)\,. 
\]
In the following, we need to assume that $\sExt_f^0(\cF,\cG)=0$ and that $s\mapsto\ext_{X_s}^1(\cF_s,\cG_s)$ is a constant function on $S$, which implies that $\sExt_f^1(\cF, \cG)$ is locally free. Furthermore, we assume that $S$ is reduced. We consider the $\IP$-bundle $\alpha\colon Y=\IP(\sExt_f^1(\cF, \cG)^\vee)\to S$ and the cartesian diagram
\[
\begin{tikzcd}
X_Y \arrow[r, "\alpha_X"] \arrow[d, "f_Y"']
& X \arrow[d, "f"] \\
Y \arrow[r, "\alpha"]
&  S\,.
\end{tikzcd}
\]

\begin{theorem}[\cite{Lange--universal}]\label{thm:Lange}
On $X_Y$, there is a short exact sequence
\begin{align}\label{eq:univext}
0\to \alpha_X^*\cG\otimes f_Y^*\reg_\alpha(1)\to \cE\to \alpha_X^*\cF\to 0 
\end{align}
which has the following universal property: 
Let $g\colon T\to S$ be an $S$-scheme and consider the cartesian diagram
\[
\begin{tikzcd}
X_T \arrow[r, "g_X"] \arrow[d, "f_T"']
& X \arrow[d, "f"] \\
T \arrow[r, "g"]
&  S\,.
\end{tikzcd}
\]
Then, for every $M\in \Pic(T)$, and every short exact sequence
\begin{equation}\label{eq:Sext}
 0\to g_X^*\cG\otimes f_T^*M\to \cD\to g_X^*\cF\to 0
\end{equation}
which is non-splitting over every $t\in T$, there is a unique classifying $S$-morphism $\beta\colon T\to Y$ such that the short exact sequences $\beta^*$\eqref{eq:univext} and \eqref{eq:Sext} are isomorphic up to the action of $\Ho^0(T,\reg_T^*)$ on short exact sequences given by multiplication of the second map of the sequences. In particular, $\beta^*\cE\cong \cD$.
\end{theorem}

\subsection{Closure of the moduli of stable line bundles}

For $o\in O$, write $\zeta_o:=C^o\cap C_o$ for the scheme-theoretic intersection. By \eqref{eq:intnumber}, this is a length 2 scheme, consisting of two reduced points if $\Gamma=\tilde A_n$, and of one fat point if $\Gamma\in \{\tilde D_n,\tilde E_6, \tilde E_7,\tilde E_8\}$. We also write $\Delta_o\subset C_o\times C_o\subset C_o\times C$ for the diagonal of $C_o$. 

\begin{lemma}
For every $o\in O$, there is a short exact sequence
 \begin{equation}\label{eq:Deltaoses}
  0\to \reg_{C_o\times C_o}\bigl(-(C_o\times \zeta_o)+\Delta_o\bigr)\to \reg_{C_o\times C}(\Delta_o)\to \reg_{C_o\times C^o}\to 0\,.
 \end{equation}
\end{lemma}

Note that the left term of the sequence is a line bundle on $C_o\times C_o$ as $\Delta_0$ is a Cartier divisor in $C_o\times C_o$. But the middle term is not a line bundle on $C_o\times C$ as $\Delta_o$ is not a Cartier divisor in $C_o\times C$.

\begin{proof}
This is just a relative version of \autoref{lem:sesCv} with a very similar proof: we have a chain of inclusions  
 \[\cI_{C_o\times C_o\hookrightarrow C_o\times C}=\reg_{C_o\times C^o}\bigl(-C_o\times \zeta_o\bigr)\subset \reg_{C_o\times C}(-\Delta_o)\subset\reg_{C_o\times C}\,.\]
Snake lemma gives a commutative diagram with exact columns and rows
\[
\begin{tikzcd}
 & 0\dar & 0\dar & \\
& \reg_{C_o\times C^o}\bigl(-C_o\times \zeta_o\bigr)\dar \rar[equal] & \reg_{C_o\times C^o}\bigl(-C_o\times \zeta_o\bigr)\dar\\
0 \rar& \reg_{C_o\times C}(-\Delta_o)\dar\rar & \reg_{C_o\times C} \dar\rar & \reg_{\Delta_o} \dar[equal]\rar &0\\
0 \rar& \reg_{C_o\times C_o}(-\Delta_o) \dar\rar & \reg_{C_o\times C_o} \dar \rar & \reg_{\Delta_o} \rar  &0\\
& 0 & 0 &
\end{tikzcd}.
\]
 Taking the dual of the left column gives the desired short exact sequence; see \autoref{lem: decomposition sequence}.
\end{proof}

Let $\cF=\pr_{13}^*\cP\bigl(\wJ\times \Delta\bigr)$ be the sheaf on $\wJ\times C\times C$ considered in \autoref{subsect:over} to give the classifying morphism $\phi\colon \wJ\times C\to \cM$. We write 
$\cF_o:=\cF_{\mid \wJ\times C_o\times C}$. 

\begin{cor}
There is a short exact sequence
\begin{equation}\label{eq:Fo}
  0\to \pr_{13}^*\cP_{\mid \wJ\times C_o}\bigl(-(C_o\times \zeta_o)+\Delta_o\bigr)\to \cF_o\to \pr_{13}^*\cP_{\mid \wJ\times C^0}\to 0\,.
 \end{equation}
\end{cor}

\begin{proof}
We apply tensor product with the line bundle $\pr_{13}^*\cP$ to $\pr_{23}^*$\eqref{eq:Deltaoses}.  
\end{proof}

\begin{lemma}\label{lem:ext1=2}
For any two line bundles $K\in \Pic(C_o)$ and $M\in \Pic(C^o)$, we have
\[
\hom_{\reg_C}(M,K)=0\quad,\quad \ext^1_{\reg_C}(M,K)=2\,.
\]
\end{lemma}

\begin{proof}
 As the intersection of the supports is the zero-dimensional scheme $\zeta_o$, we may assume that both line bundles are trivial: $K\cong \reg_{C_o}$, $M\cong \reg_{C^o}$. For support reasons, we have $\sHom(\reg_{C^o},\reg_{C_o})=0$.  Considering the exact sequence
 \[
  0\to \reg_{C_o}(-\zeta_o)\to \reg_C\to \reg_{C^o}\to 0
 \]
of \autoref{lem: decomposition sequence} and applying $\sHom(\_,\reg_{C_o})$ gives
\[
 0\to \reg_{C_o}\to \reg_{C_o}(\zeta_o)\to \sExt^1(\reg_{C^o},\reg_{C_o})\to 0\,.
\]
Hence, $\sExt^1(\reg_{C^o},\reg_{C_o})\cong \reg_{\zeta_o}$. The local-to-global-Ext spectral sequence now gives the assertion as $\ho^0(\reg_{\zeta_o})=2$. 
\end{proof}

\begin{lemma}\label{lem:extensionunique}
Let $K\in \Pic(C_o)$ and $M\in \Pic(C^o)$, and consider two extensions 
\begin{align}
0\to K\to F_1\to M\to 0\,,\label{eq:ext_1}  \\
0\to K\to F_2\to M\to 0 \label{eq:ext_2}
\end{align}
of coherent sheaves on $C$. Then, $F_1\cong F_2$ if and only if \eqref{eq:ext_1} and \eqref{eq:ext_2} agree up to the  
action of $\IC^*=\Ho^0(\reg_C^*)$ on extensions.
\end{lemma}

\begin{proof}
By definition, the $\IC^*$-action only changes the second map of the extensions, but not the middle term. This already gives the \emph{if} part of the statement.

Let now $\phi\colon F_1\to F_2$ be an isomorphism. We have $\pure(F_{1\mid C^o})\cong M\cong \pure(F_{2\mid C^o})$ by \autoref{lem:quotientsarerestrictions}. Hence, we get an isomorphism of short exact sequences 
\[
\begin{tikzcd}
0\rar& K \arrow{r}\dar{\cong} & F_1\dar{\phi}\rar & M  \dar{\pure(\phi_{\mid C^o})} \rar &0 \\
0\rar& K \arrow{r}& F_2\rar & M  \rar &0\,.  
\end{tikzcd} 
\]
Since $K$ and $M$ are simple (for the latter, this follows by \autoref{lem:1con}), it follows that \eqref{eq:ext_1} and \eqref{eq:ext_2} agree up to a non-zero scalar.
\end{proof}

We write $\Pic_b(C^o)$ for the connected component of $\Pic(C^o)$ consisting of line bundles $M$ on $C^o$ with $\chi(M_{\mid C_v})=b_v$ for all $v\in V\setminus \{o\}$. 

 \begin{lemma}\label{lem:PicCo}
 The morphism 
 \[\pi_o\colon \Pic_b(C^o)\to \prod_{v\in V\setminus\{o\}} \Pic_{b_v}(C_v)\cong \prod_{j\in O\setminus \{o\}} \Pic_{b_j}(C_j)\,, \] given by restriction of line bundles to the irreducible components, is an isomorphism.
\end{lemma}

\begin{proof}
This is similar to the proof of \autoref{lem:PicCx}. By \autoref{cor:Piciso}, it is enough to prove that $g(C^o)=\sum_{j\in O\setminus \{o\}}g(C_j)$. 
By \eqref{eq:intnumber}, we have 
\begin{equation}\label{eq:chiCo}
 \chi(\reg_{C^o})=\chi(\reg_C)-\chi(\reg_{C_o})+2\,.
\end{equation}
As we have seen before, the fact that $C^o$ is a numerically $1$-connected Gorenstein curve implies  $\ho^0(\reg_{C^o})=1=\ho^0(\reg_C)=\ho^0(\reg_{C_o})$; see \autoref{lem:1con} and \cite[Thm.~3.3]{CFHR99}.
Combining this with \eqref{eq:chiCo} and \autoref{lem:gC} gives
\[
 g(C^o)=g(C)-g(C_o)-1=1+\sum_{j\in O}g(C_j) -g(C_o) -1=\sum_{j\in O\setminus \{o\}}g(C_j)\,.\qedhere
\]
\end{proof}

We set 
\[
J_o:=\Pic_{b_o-1}(C_o)\times \Pic_{b}(C^o)\,.
\]
Note that, by \autoref{lem:PicCo}, $J_o$ is canonically isomorphic to $\Pic_{b_o-1}(C_o)\times \prod_{j\in O\setminus \{o\}} \Pic_{b_j}(C_j)$, hence non-canonically isomorphic to $J=\prod_{j\in O} \Pic_{b_j}(C_j)$. We write the projections as 
\[
 p_o\colon J_o\to \Pic_{b_o-1}(C_o) \quad,\quad q_o\colon J_o\to \Pic_{b}(C^o)
\]
and denote universal families of $\Pic_{b_o-1}(C_o)$ and $\Pic_{b}(C^o)$ by $\cP_o$ and $\cP^o$, respectively. We consider the two sheaves
\[
\cQ_o:=i_*(p_o\times \id_{C_o})^*\cP_o\quad , \quad \cQ^o:=j_*(q_o\times \id_{C^o})^*\cP^o 
\]
on $J_o\times C$ where $i\colon J_o\times C_o\hookrightarrow J_o\times C$ and $j\colon J_o\times C^o\hookrightarrow J_o\times C$ are the embeddings.

For any point $t\in J_o$, the fibre 
$(\cQ_o)_t$ is a line bundle on $C_o$ and $(\cQ^o)_t$ is a line bundle on $C^o$.  
Hence, \autoref{lem:ext1=2} shows that $\ext^1_{C}((\cQ^o)_t, (\cQ_o)_t)=2$ is a constant function in $t\in J_o$, and $\sExt^0_{f_o}(\cQ^o, \cQ_o)=0$, where $f_o\colon J_o\times C\to J_o$ denotes the projection to the first factor. Hence, the assumptions of \autoref{subsect:universalext} are fulfilled, which gives that 
\[
 \alpha_o\colon Y_o:=\IP\bigl(\sExt_{f_o}^1(\cQ^o, \cQ_o)^\vee\bigr)\to J_o
\]
is a $\IP^1$-bundle satisfying the universal property described in \autoref{thm:Lange}. 

\begin{prop}\label{prop:Yo}
 There is a commutative diagram 
\begin{equation}\label{eq:beta}
 \begin{tikzcd}
&  & \cM\\
\wJ\times C_o \arrow{rr}{\beta_o} \arrow{rru}{\phi_o} \arrow{rrd}{g_o} & {} & Y_o \dar{\alpha_o} \uar{\iota_o} \\
&   & J_o 
\end{tikzcd} 
\end{equation}
where $\phi_o=\phi_{\mid\wJ\times C_o}$, $g_o(L,x)=\bigl(L_{\mid C_o}(-\zeta_o+x),L_{\mid C^o}\bigr)$, $\beta_o$ is surjective, and $\iota_o$ is a closed embedding. 
\end{prop}

\begin{proof}
 The map $\beta_o$ is the classifying morphism for the extension \eqref{eq:Fo}. Let us explain this in more detail. The sheaf $\pr_{13}^*\cP_{\mid \wJ\times C_o}\bigl(-(C_o\times \zeta_o)+\Delta_o\bigr)$ is a family of line bundles on $C_o$ of Euler characteristic $b_o-1$, parametrised by $\wJ\times C_o$. Let $u\colon \wJ\times C_o\to \Pic_{b_o-1}(C_o)$ be the classifying morphism. The sheaf $\pr_{13}^*\cP_{\mid \wJ\times C^0}$ is a family of line bundles on $C^o$ whose restrictions to $C_j$ for $j\in O\setminus\{o\}$ are of Euler characteristic $b_j$, parametrised by $\wJ\times C_o$. Let $v\colon \wJ\times C_o\to \Pic_{b}(C^o)$ be the classifying morphism. Let $g_o=(u,v)\colon \wJ\times C_o\to J_o$. On closed points, it is given by $g_o(L,x)=\bigl(L_{\mid C_o}(-\zeta_o+x),L_{\mid C^o}\bigr)$, as asserted.
Furthermore, \eqref{eq:Fo} is isomorphic to 
\[
  0\to (g_o\times \id_C)^*\cQ_o\to \cF_o\to (g_o\times \id_C)^*\cQ^o\to 0\,.
\]
Hence, \autoref{thm:Lange} gives a classifying morphism $\beta_o$ making the lower part of diagram \eqref{eq:beta} commute. 

The fibres of $\alpha_o$ are irreducible projective curves (namely, isomorphic to $\IP^1$), and the fibres of $g_o$ are irreducible (namely, isomorphic to $G\times C_o$). Hence, for the surjectivity of $\beta_o$, it suffices to show that no fibre of $g_o$ is contracted by $\beta_o$ to a point. To see this, let $(K,M)\in J_o$, let $x\in \zeta_o$ be an intersection point of $C_o$ with another component of $C$, and let $L\in \wJ$ be any line bundle with $L_{\mid C_o}(-\zeta_o+x)\cong K$ and $L_{\mid C^o}\cong M$, hence $g_o(L,x)=(K,M)$. A second point in the same fibre of $g_o$ is $\bigl(L', y\bigr)$ for some $y\in C_o\setminus \zeta_o$ and some line bundle $L'\in \wJ$ with $L'_{\mid C_o}(-\zeta_o+y)\cong K$ and $L_{\mid C^o}\cong M$. But $\cF_{(L,x)}\cong L(x)\not\cong L'(y)\cong \cF_{(L',x)}$ as the right side is a line bundle, while the left side is not. Hence, also $\beta_o(L,x)\neq \beta_o(L',y)$; see the \emph{if} part of \autoref{lem:extensionunique}.   

The surjectivity of $\beta_o$ implies that every fibre of the middle term of the universal extension on $Y_o$ occurs as a fibre of $\cF_o$. In particular, every fibre of the middle term of the universal extension is a stable sheaf of type $m$. Hence, we get the morphism $\iota_o\colon Y_o\to \cM$ as the classifying morphism for this middle term.

The injectivity of $\iota_o$ is the \emph{only if} part of \autoref{lem:extensionunique}. By \autoref{lem:dphi}, the differential of $\iota_o$ is injective in every point. Hence, $\iota_o$ is a closed embedding.
\end{proof}

\subsection{Description of the moduli space}

In this subsection, we collect our results. \autoref{thm: easy moduli} stated in the introduction follows immediately from a more precise
result that we now formulate. 
\begin{theorem}\label{thm: full moduli}
In the above situation the following hold:
 \begin{enumerate}
  \item The locus in $\cM_\red$ parametrising singular stable sheaves (i.e.\ sheaves which are not line bundles) is isomorphic to $J\times (C_{\sing})_{\mathsf{red}}$, with the isomorphism given by $L(x)\mapsto\bigl((L_{\mid_{C_o}})_{o\in O}, x\bigr)$.  
  \item The locus in $\cM_\red$ parametrising stable line bundles consist of $|O|$ connected components $U_o$, each of which is a $G$-torsor over $J$ where $G= \IG_m$ if $\Gamma=\tilde A_n$, and $G=\IG_a$ if $\Gamma\in \{\tilde D_n, \tilde E_6,\tilde E_7,\tilde E_8\}$.
  \item Let $Y_o$ denote the closure of $U_o$ in $\cM_\red$. Then, $Y_o$ is a $\IP^1$-bundle over $J$ for every $o\in O$. The complement of $U_o$ in $Y_o$ is a disjoint union of sections of this $\IP^1$-bundle, one section $S^o_x$ for every intersection point $x$ of $C_o$ with another component of $C$. The sections $S^o_x$ parametrise stable sheaves which are not line bundles in $x$.
  \item Let $\Gamma=\tilde A_n$. Then, for an intersection point $x=C_o\cap C_{o'}$, the components $Y_o$ and $Y_{o'}$ of $\cM_{red}$ intersect transversally in the sections $S^o_x$ and $S^{o'}_x$.
  \item Let $\Gamma\in \{\tilde D_n, \tilde E_6,\tilde E_7,\tilde E_8\}$. Then, for every $o\in O$, there is a single intersection point $x$ with another component $C_i$ which is of multiplicity $2$. Then, the component $Y_o$, and the component $Y_i:=J\times C_i$ (parametrizing stable sheaves which are singular in some point of $C_i$) intersect transversally in $S_x^o=J\times \{x\}$.
 \end{enumerate}
\end{theorem}

\begin{proof}
 By \autoref{prop:Csingred}, we have a closed embedding \[\psi\colon J\times (C_\sing)_{\mathsf{red}}\hookrightarrow \cM\quad,\quad\bigl((L_{o})_{o\in O}, x\bigr)\mapsto L(x)\,,\]
 where $L\in \Pic(C)$ is any line bundle with $L_{\mid C_o}=L_o$ (by \autoref{prop:unique}, it does not matter which line bundle $L$ with this property we pick). By \autoref{prop:classOx} its image is exactly the locus of singular stable sheaves. So we proved part (i). 
 
 For part (iii), we identify the $\IP^1$-bundles 
\[Y_o\xrightarrow{\alpha_o} J_o\]
of \autoref{prop:Yo} with their images under the closed embeddings $\iota_o\colon Y_o\hookrightarrow \cM$. 
We can consider $Y_o$ as a $\IP^1$-bundle over $J$ by choosing any isomorphism $J_o\cong J$.

By \autoref{prop:Yo}, we have $Y_o=\phi(\wJ\times C_o)$. Hence, the locus of singular sheaves in $Y_o$ is the union $\coprod_x \phi(\wJ\times \{x\})$ over all intersection points $x$ of $C_o$ with another component. 
Let $x$ be such an intersection point. The restrictions to $\wJ\times \{x\}$ of all three morphisms $\phi_o,\beta_o,g_o$ in the diagram \eqref{eq:beta} are $G$-invariant, hence factor over the quotient $J\times \{x\}$. The morphism $g_{o\mid \wJ\times \{x\}}$ induces the morphism
\[
 \bar g_o^x\colon J\times \{x\}\to J_o\quad,\quad (L_j)_{j\in O}\mapsto \Bigl( L_o, \pi_o^{-1}\bigl((L_j)_{j\in O\setminus \{0\}}\bigr)\Bigr)\,.
\]
This is an isomorphism; see \autoref{lem:PicCo}. Hence, the locus $S_x^o$ of stable sheaves which are singular at $x$ is a section of $Y_o$.

If $\Gamma \in \{\tilde D_n, \tilde E_6,\tilde E_7,\tilde E_8\}$, there is only one intersection point $x$ of $C_o$ with other components of $C$. Hence, $U_o=Y_o\setminus S_x^o$ is a $\IG_a$-bundle over $J$. 

If $\Gamma = \tilde A_n$ there are two intersection points $x,y$ of $C_o$ with other components of $C$. Hence, $U_o=Y_o\setminus (S_x^o\cup S_y^o)$ is a $\IG_m$-bundle over $J$. Thus, we have also proved part (ii). 

Finally, the transversally statement in parts (iv) and (v) follows from \autoref{lem:dphi}. Indeed, let $x=C_o\cap C_j$, be an intersection point. We have $Y_o=\phi(\wJ\times C_o)$. If we are in the $\Gamma=\tilde A_n$ case, where $j\in O$, also $Y_j=\phi(\wJ\times C_j)$. If we are in the $\Gamma \in \{\tilde D_n, \tilde E_6,\tilde E_7,\tilde E_8\}$ case, we still write $Y_j:=\phi(\wJ\times C_j)=\psi(J\times C_j)$ for the component of $\cM_\red$ intersecting $Y_o$. We have a direct sum decomposition 
\[
 T_{\wJ\times C}(L,x)\cong T_{\wJ}(L)\oplus T_C(x)\cong K\oplus T_J(\pi(L))\oplus T_{C_o}(x)\oplus T_{C_j}(x)\,, 
\]
where $K=\ker(d\phi_{(L,x)})$; see \autoref{lem:dphi}. It follows that 
\[
T_{Y_o}(L(x))=d\phi(K\oplus T_J(\pi(L))\oplus  T_{C_o}(x))= d\phi(T_J(\pi)\oplus  T_{C_o}(x))
\]
and
\[
T_{Y_j}(L(x))=d\phi(K\oplus T_J(\pi(L))\oplus  T_{C_j}(x))= d\phi(T_J(\pi)\oplus  T_{C_j}(x))
\]
are non-identical subspaces of $T_{\cM_\red}(L(x))$, which means transversality of $Y_o$ and $Y_j$.
\end{proof}

\section{Further Remarks}\label{sect:furtherremarks}

\subsection{Codimension one strata in linear systems on K3 surfaces}\label{sect: linear systems}

Extended ADE curves occur as generic singular members of linear systems on K3 surfaces as we will now explain.
Throughout this section, let $X$ be a K3 surface.
Let $L$ be a line bundle on $X$ with at least two sections such that the general element in $|L|$ is a smooth connected curve of genus $g$. 
\begin{lemma}\label{lem:K3linsys}
The following properties hold:
\begin{enumerate}
 \item $L^2 = 2g-2 \geq 0$ and $\dim |L| = g$;
 \item $L$ is nef and $|L|$ has no base points;
  \item if $L^2 >0$ then $L$ is big and nef;
\item if $L^2 = 0 $ then $L$ is an elliptic pencil.
  \end{enumerate}
\end{lemma}
\begin{proof}
These are well-known results, all of which can be found in the textbook \cite{HuybrechtsK3}. More precisely, see Cor. 2.1.5, Lem.\ 2.2.1, and Rem.\ 8.2.13 of \emph{loc.\ cit.} 
\end{proof}

\begin{prop}\label{prop:ADEdivisor}
Let $C\hookrightarrow X$ be an extended ADE curve in $|L|$. Then, in the linear system $|C|=|L|$, the locus of curves which are still extended ADE curves of the same type (i.e.\ the same intersection graph $\Gamma$ and the same genera of the reduced components) as $C$ is of codimension $1$. 
\end{prop}

\begin{proof}
We consider the natural morphism 
\[
 f\colon \prod_{o\in O}|C_o|\to |C|,
\]
whose image consists of all curves which arise from $C$ by moving the reduced components $C_o$ of $C$ in their linear systems. This map is generically injective and a general member of its image is still an extended ADE curve of the same type as $C$. By \autoref{lem:gC} together with \autoref{lem:K3linsys}(i), the dimension of the domain of $f$ is $\dim|C|-1$, which proves the claim. 
\end{proof}

In the following, we will see that, under a positivity assumption, a kind of converse of \autoref{prop:ADEdivisor} holds. Namely, reducible members of the linear system $|L|$ which occur in codimension 1 must be extended ADE curves, with only one irrational component.

The linear system  $|L|$ is naturally stratified into equisingular strata \cite{DS17, Wahl74}. By our assumptions the general element is a smooth curve, so the discriminant $\Delta\subset |L|$ parametrising singular curves in the linear system is a proper closed subset. We want to investigate, what kind of curves occur at the general point of a component of the discriminant $V\subset |L|$ of codimension one. 

If $L^2 = 0$, then $|L|$ is an elliptic pencil and in codimension one we have the possible singular fibres classified by Kodaira. Thus in the following we assume that $L^2 = 2g-2>0$. 

The following was explained to us by Andreas Knutsen:
\begin{prop}\label{prop: codim 1 irreducible strata}
 Assume that the integral curve  $C$ is general in a codimension one stratum $V\subset|L|$. Then $C$ has a unique node.
\end{prop}
\begin{proof}
 Let $\tilde C \to C$ be the normalisation. By \cite[Prop.~4.5]{DS17}, the family of curves of geometric genus $g(\tilde C)$ is of dimension $g(\tilde C) = \dim V = \dim|L| -1$ and by \cite[Prop.~4.6]{DS17} the generality of $C$ implies that $C$ has ordinary singularities, i.e., $\tilde C \to X$ is an immersion. 
 Then $C$ can only have one node, as $p_a(C) = g(\tilde C) +1$. 
\end{proof}

We were also made aware by Justin Sawon that \autoref{prop: codim 1 irreducible strata} was already proven in \cite[Lemma 2.4]{Sawon--LagrangianJac}. As the proof above differs a lot from the one in \emph{loc.\ cit}., we decided to keep it in the text. 

Let us consider an example to illustrate how reducible curves can arise in codimension one strata of $|L|$. This example was the original motivation for the present work. 
\begin{example}\label{exam: double cover}

 Let $B\subset \IP^2$ be a plane sextic, possibly with ADE singularities. Consider the double cover $\bar X$ branched over $B$ and its minimal resolution
 \[
  \begin{tikzcd}
   X \rar{\eta} \arrow{dr} & \bar X \dar{\pi}\\& \IP^2
  \end{tikzcd}
 \]
so that $X$ is a smooth K3 surface and $\bar X$ possibly has ADE singularities. Clearly, $X = \bar X$ if and only if $B$ is smooth.

Let $\bar L = \pi^*\ko_{\IP^2}(1)$ and $L = \eta^*\bar L$. By the projection formula we have $|dL| = \eta^*\pi^*|\ko_{\IP^2}(d)|$ for $d = 1,2$.

The pullback of a line $M\in |\reg_{\IP^1}(1)|$ will be reducible in $\bar X$ if and only if $B|_M$ is divisible by $2$ as a divisor on $M$, i.e.\ if $M$ is a triple tangent of $B$. Such lines lead to singularities in the dual curve $B^\vee$ and in particular there are only finitely many. 
  Therefore, all but finitely many curves in $|\bar L|$ are irreducible.
  
  If the branch locus is smooth, then all strata of codimension one in $|L|$ parametrise curves with one node as in  \autoref{prop: codim 1 irreducible strata}. These correspond to tangent lines to $B$ and are thus parametrised by an open subset of the dual curve $B^\vee$. 
 
 If $p$ is a singular point of the branch curve $B$, then the pencil of lines through $p$ pulls back to a pencil of curves on $\bar X$ passing through the ADE surface singularity. On $X$ we get a pencil of curves containing the exceptional divisor of the singularity as a fixed part. The general such curve is an extended ADE curve where the only non-rational component is the strict transform of the pullback of the line, which is a smooth elliptic curve.

 So in this example the discriminant consists of the dual curve $B^\vee$ together with one line for each singular point of $B$.

 Now consider the linear system $|2L|$. Here we find three different strata in codimension one:
 \begin{enumerate}
  \item Pullbacks of conics that are tangent to $B$ in a point give irreducible curves with a node. 
 \item Pullbacks of general conics passing through a singular point of $B$ give curves containing an ADE configuration, where the only non-rational curve is the strict transform of the pullback of the conic, which  is a smooth curve of genus 4. 
 \item Pullbacks of the sum of two general lines give an $\tilde A_1$-curve where both components are smooth of genus two.
 \end{enumerate}

 Note that the last type of component occurs even when $B$ is smooth. 
\end{example}

We take away the impression that components as in (i) will always exist and components of type (ii) will exist if $L$ is big and nef but not ample. The extra component like in (iii) however needs particular reasons to exist. 

\begin{prop}\label{prop: codim 1 strata very positive}
 Let $\bar X$ be a possibly singular K3 surface with ADE singularities, and let $\eta \colon X \to \bar X$ be the minimal resolution. Let $\bar L$ be a line bundle on $\bar X$ such that
 \begin{enumerate}
  \item $\bar L$ is very ample;
  \item $\bar L^2 \geq 12$;
  \item for any curve $D\subset \bar X$ we have $\bar L.D>8$. 
 \end{enumerate}
Then the general curve in a codimension one stratum of $|L| = |\eta^*\bar L|$ is either an integral curve with one node or the pullback of a general curve in $|\bar L|$ passing through one ADE singularity of $\bar X$. The latter are extended ADE-curves where all but one components are $-2$ curves. 
 \end{prop}
  We  believe the positivity conditions could be sharpened  by adapting the techniques of \cite{DS17} to reducible curves and of \cite{Knutsen} to singular K3 surfaces.

  \begin{remark}
Note that, if $L$ itself is ample, we cannot have an extended ADE curve with a $(-2)$-component as a member of $|L|$, as the restriction of $L$ to the $(-2)$-component would be trivial. Hence, for $L$ ample and sufficiently positive, \autoref{prop: codim 1 strata very positive} says that singular members of $|L|$ in codimension 1 are integral with one node.  
\end{remark}
 
\begin{proof}
 Let $x\in \bar X$ be a smooth point and let $Z_x$ be the subscheme defined by $\ko_{\bar X}/\mathfrak m_x^2$. 
 By \cite[Propositions 3.1 and 3.7]{Knutsen} the map $H^0( X, L) \to H^0(X, L\tensor \ko_{Z_x+y})$ is surjective for any $y\in X\setminus \{x\}$. In other words, the linear system 
 $|\mathfrak m_x^2 \bar L|\subset |\bar L|$ has $x$ as its only base point. By Bertini, the general curve in $|\bar L|$ which is singular at $x$ is irreducible and has a unique singular point. By \autoref{prop: codim 1 irreducible strata} the singularity is a node.
 
 Now consider curves passing through a singular point $\bar x\in \bar X$. Then, since $|\bar L|$ separates points, the linear system $|\mathfrak m_{\bar x}\bar L|\subset |\bar L|$ has no other base points. Thus, by Bertini, the general such curve is smooth outside $\bar x$ and irreducible. It is clear, that the pullback of  a general hyperplane passing through $\bar x$ is an extended ADE curve on $X$.
 \end{proof}

\begin{example}\label{exam: double cover revisited}
 Let $X$ be as in Example \ref{exam: double cover} and $L = \eta^* \pi ^*\ko_{\IP^2}(d)$. Then the numerical conditions of \autoref{prop: codim 1 strata very positive} are satisfied for $d\geq 5$ by the projection formula.
 
 For $d=1$ the numerical conditions are not satisfied, but the conclusion of \autoref{prop: codim 1 strata very positive} still holds for the linear system $|L|$, while for $d=2$ this is not the case. 
 
 We did not consider the remaining cases $d =3,4 $ in detail.
 
\end{example}

\begin{example}\label{ex:doublecov}
Let $\phi\colon \bar X \to Q\subset \IP^3$ be a double cover of the quadric cone in $\IP^3$ branched over a general quartic section. Since no branching occurs at the vertex of the cone, $\bar X$ contains two $A_1$-singularities both mapping to the vertex. 

Let $L = \phi^*\ko_{\IP^3}(1)|_Q$. By the projection formula the elements of $|L|$ are the pullbacks of hyperplane section of $Q$. The general element is a smooth curve of genus three. 
We now describe the general members of the subsystem of codimension one  coming from hyperplanes containing the vertex. 

If $H$ is a general  hyperplane through the vertex of the cone, the intersection $H\cap Q$ consists of two lines passing through the vertex. Therefore, the pullback $\bar X\supset \phi^*H = C_1+ C_2$ where $C_i$ is a smooth elliptic curve passing through both singular points. 
Pulling back further to the minimal resolution $f\colon X \to \bar X$, which is a K3 surface, we have
\[ f^*\phi^*H = E_1 + C_1 + E_2 + C_2,\]
an $\tilde A_4$ configuration consisting of two elliptic and two $-2$ curves. 

Note that this extended ADE-configuration does not conform to the simple pattern found in 
\autoref{prop: codim 1 strata very positive} (more than one component is non-rational), which does not apply because $L$ is not sufficiently positive. 
\end{example}

\autoref{exam: double cover revisited} shows that  the positivity conditions are not necessary for the conclusion of \autoref{prop: codim 1 strata very positive} to hold. 

But even if the conclusion of \autoref{prop: codim 1 strata very positive} does not hold for a linear system $|L|$ on a K3 surfaces, in the examples we looked at it is still true that the general curve in a codimension one stratum is always either irreducible with one node as in \autoref{prop: codim 1 irreducible strata} or an extended ADE curve, with possibly more than one non-rational component.
 It is therefore legitimate to speculate that this is always the case.

To apply our description of  moduli spaces of sheaves to different  curves in a linear system $|L|$ we should pick a global polarisation on the K3 surface $X$. 
\begin{prop}\label{prop: global polarisation possible}
 Assume that on a K3 surface $X$ we have a linear system $|L|$ satisfying the conclusion of \autoref{prop: codim 1 strata very positive}.
 
 Then there exists  a polarisation $H$ on $X$ such that, for any reducible curve $C$ which is general in codimension one in $|L|$, the restricted polarisation satisfies \autoref{ass:H}.
 \end{prop}
\begin{proof}
 We consider the contraction $\eta \colon X \to \bar X$ from \autoref{prop: codim 1 strata very positive} and for every singular point $p\in \bar X$ a fundamental cycle $Z_p$, such that $Z_p$ is supported on the preimage of $p$ and $-Z_p.E >0$ for any curve contracted to $p$. 
 
 Now choose any ample line bundle $\bar H$ on $\bar X$ and consider the $\IQ$-divisor
 \[ H = \eta^* \bar H - \sum_{p\in \Sing{\bar X}} \epsilon_p Z_p\,.\]
If we choose all $\epsilon _p$ sufficiently small, we achieve \autoref{ass:H} for all curves of the form $\eta^*D$ where $D\in |\bar L|$ is a general curve passing through some $p\in \Sing{\bar X}$ by the same argument as in \autoref{prop:Hone}.
\end{proof}

 \begin{example}\label{exam: no global polarisation}
  We consider the double cover of the plane as in  \autoref{exam: double cover} for a very general sextic branch curve $B$. Then $X$ is smooth and has Picard rank one by a result of Cox \cite{cox89}, so up to multiples the only polarisation is given by $L = \pi^* \ko(1)$.
  
  The only reducible singular curves occurring in codimension one in $|2L|$ are the pullbacks of two lines forming an $\tilde A_1$ configuration. \autoref{ass:H} becomes
     \[ \left\lceil \frac{1}{2} \chi\right\rceil + \left\lceil\frac{1}{2} \chi\right\rceil = \chi +1 \]
 which is satisfied if and only if $\chi$ is odd.
 In the case that $\chi$ is even, \autoref{thm: easy classification} does not hold anymore. Indeed, we then have properly semi-stable line bundles, namely those whose restriction to one component has Euler characteristic $\frac \chi2$ while the restriction to the other component has Euler characteristic $\frac \chi2 +2$.
 \end{example}

\subsection{Beauville--Mukai systems}\label{subsect:family}

We now want to let our moduli spaces vary along the linear system $|C|=\IP\bigl(\Ho^0(\reg_X(C)) \bigr)$. More precisely, let $C\hookrightarrow X$ be an extended ADE curve inside some smooth projective surface $X$, and $H$ a polarisation on $X$. We write $\cN:=\cM^H_{0,[C], \chi}(X)$ for the moduli space of semi-stable sheaves on $X$ of rank $0$, first Chern class $[C]$, and Euler characteristic $\chi$. As every $[F]\in \cM^H_{0,[C],\chi}$ represents a purely one-dimensional sheaf, we have a locally free resolution
\[
 0\to E_1\xrightarrow s E_0\to F\to 0\,.
\]
From this, we get an element $Z(\det(s))\in |C|$ not depending on the chosen resolution. This gives a well-defined morphism $\cN\to |C|$, and our moduli space $\cM:=\cM^H_{m,\chi}(C)$ of stable sheaves of type $m$ and Euler characteristic $\chi$ is the fibre $\cN_C$ over $C\in |C|$. Hence, if $H_{\mid C}$ satisfies \autoref{ass:H}, we have a description of this fibre of $\cN\to |C|$.

Our main motivation for this work was to study these fibres when $X$ is a K3 surface. In this case, if the Mukai vector $(0,[C],\chi)$ is primitive, then $\cN$ is an irreducible holomorphic symplectic variety, and $\cN\to |C|$ is a Lagrangian fibration, called a \textit{Beauville--Mukai integrable system}.

From now on, until the end of the next subsection, let $X$ be a K3 surface.
By \autoref{prop:ADEdivisor} together with \autoref{prop: global polarisation possible}, our main result \autoref{thm: full moduli} describes certain generic singular fibres of $\cN\to |C|$, i.e.\ singular fibres which occur in codimension 1.  
By \autoref{prop: codim 1 strata very positive}, under sufficient positivity assumptions on $L=\reg_X(C)$, all generic singular fibres are of this form, or compactified Jacobians of curves with one node. In the latter case, a description of the compactified Jacobian is well-known; see e.g.\ \cite[Ex.\ 1]{Sawon--disclocus}. Namely, let $C$ be a curve with one node, and $\tilde C$ its normalisation. Then the moduli space $\bar J_\chi(C)$ of rank one torsion-free sheaves on $C$ of Euler characteristic $\chi$ is a $\IP^1$-bundle over the Jacobian of $\tilde C$ with two sections of this bundle glued via a translation. In particular, if $\tilde C$ is rational, then $\bar J_\chi(C)$ is again a rational nodal curve.

There are strong results on the structure of generic singular fibres of arbitrary holomorphic Lagrangian fibrations due to Hwang and Oguiso \cite{Hwang-Oguiso--charfol}, \cite{Hwang-Oguiso--multiple}; compare also \cite{Matsu--Lagrangian}. In the next subsection, we explain how our description of $\cM$ fits into these general results.

\subsection{Characteristic cycles}

Let $M$ be a holomorphic symplectic variety of dimension $2n$, and $f\colon M\to \IP^n$ a holomorphic Lagrangian fibration. Then, by \cite[Prop.\ 3.1]{Hwang-Oguiso--charfol}, the discriminant locus 
\[
 \Delta=\bigl\{b\in \IP^n\mid M_b \text{ is singular}\bigr\}
\]
is a hypersurface in $\IP^n$. Furthermore, \cite{Hwang-Oguiso--charfol}, \cite{Hwang-Oguiso--multiple} give strong results on the geometry of the general fibres over the components of $\Delta$. We only recall (parts of) the results for \emph{reducible} generic singular fibres, as this is the case relevant to us. Let $b\in \Delta$ be a general point in some component such that $M_b$ is reducible. Let $\{Y_w\}_{w\in W}$ be the family of the irreducible components of $(M_b)_\red$. Then the $Y_w$ are smooth of Albanese dimension $n-1$, and the Albanese map $\alpha_w\colon Y_w\to \Alb(Y_w)$ is a $\IP^1$-bundle.

Now, one can define \emph{reduced characteristic cycles} on $(M_b)_{\red}$. Namely, let $\sim$ denote the equivalence relation on closed points of $(M_b)_{\red}$ generated by the relation $P\sim Q$ for two points on the same $\IP^1$-fibre of the same component $\alpha_w\colon Y_w\to \Alb(Y_w)$ of $M_b$. A \emph{reduced characteristic cycle} is then an equivalence class under this equivalence relation.

Intuitively, we form a reduced characteristic cycle by starting at any point $P$ of $M_b$, say with $P\in Y_w$. Then we can walk along the $\IP^1$-fibre of $Y_w$ until we reach the intersection locus with some other $\IP^1$-bundle component of $M_b$. Then we can walk along the $\IP^1$-fibre of the other bundle containing the intersection point, and so on. The configuration of smooth rational curves that we obtain is the reduced characteristic cycle of $(M_b)_{\red}$.
One then defines the (not necessarily reduced) \emph{characteristic cycle} of $M_b$ by counting every $\IP^1$-fibre of the reduced characteristic cycle with the multiplicity of the component $Y_w$ it lies in.
It is then proven in \cite{Hwang-Oguiso--charfol}, \cite{Hwang-Oguiso--multiple} that all characteristic cycles of $M_b$ are isomorphic, and of one of the following types:
\begin{enumerate}
 \item a reducible fibre of an elliptic fibration, i.e.\ of type III, type IV, or an extended ADE curve where all components are rational,
 \item of type $A_{\infty}$ which means an infinite chain of smooth rational curves. 
\end{enumerate}
Let us see how the characteristic cycles look like in the case of our moduli spaces $\cM$ of sheaves of type $m$ on extended ADE curves. In the case that $\Gamma\in \{\tilde D_n,\tilde E_6,\tilde E_7,\tilde E_8\}$, it is quite obvious that \autoref{thm: easy moduli} implies that the \emph{reduced} characteristic cycles of $\cM_{\red}$ are of the same type $\Gamma$ (without multiplicities). Then the results of \cite{Hwang-Oguiso--charfol}, \cite{Hwang-Oguiso--multiple} say that also the non-reduced characteristic cycles must be of the same type $\Gamma$, but now with multiplicities given by the labels of $\Gamma$. This gives further evidence for \autoref{Conj:multi}. See \cite{Hellmann} for the computation of multiplicities of the components of moduli space of sheaves on a curve in a somewhat related special case. 

In the case $\Gamma=\tilde A_n$, we have $\cM=\cM_{\red}$. However, there is another subtlety making things complicated. Namely, if we walk around the $\tilde A_n$-configuration along the fibres to form the characteristic cycle as described above, in general, after walking one round, we will end up in another fibre than the one in which we started. This means that it can happen that the characteristic cycle is not of type $\tilde A_n$ but of type $\tilde A_{kn}$ for some $k\in \IN$, or even $k=\infty$. This phenomenon is well-known in other examples of Lagrangian fibrations; see e.g.\ the pictures in the introduction of \cite{Sawon--sing}. In Lemma 8 and the remark following its proof in \emph{loc.\ cit.}\ the characteristic cycles for moduli spaces of rank 1 sheaves over curves with one node are described. 

In the following, we describe the characteristic cycles in the case of an extended ADE curve of type $\Gamma=\tilde A_n$.
We make an identification $V=\{0,1,\dots, n\}$ in such a way that $C_{v-1}$ and $C_{v}$ as well as $C_n$ and $C_0$ intersect. We write the intersection points as $x_v:=C_{v-1}\cap C_v$ and $x_0:=x_{n+1}:=C_{n}\cap C_0$.

\vspace*{1ex}
 
\begin{center}
\includegraphics[scale=0.5]{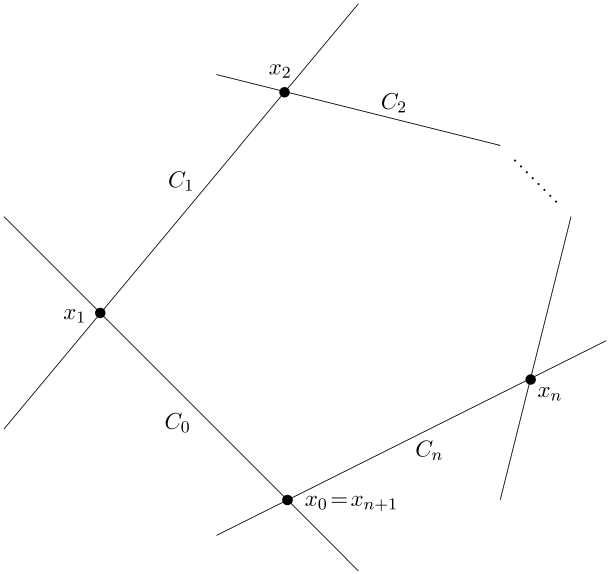}
\end{center}

\vspace*{2ex}

Let $\bigl(L_u\in \Pic_{b_u}(C_u)\bigr)_{u\in V}$ be a collection of line bundles. We set
\begin{equation}\label{eq:L'}
 L_u':=L_u(x_u-x_{u+1})\quad \text{for $u\in V$}\,.
\end{equation}
Furthermore, for $v=0,\dots,n+1$, let $F^v$ be the unique (see \autoref{prop:unique}) stable sheaf on $C$ which is special in $x_v$ and satisfies
\[
 (F^v)_j=\pure(F^v_{\mid C_j})\cong \begin{cases}
                                    L_j' \quad&\text{for $j=0,\dots, v-1$},\\
                                    L_j \quad&\text{for $j=v,\dots, n$}.
                                    \end{cases}
\]
For $v=1, \dots, n$, we can write $F^v$ as an extension in two different ways:
\begin{align}
 0\to L_{v-1}'(-x_{v-1})\to F^v\to \cL_{v-1}(L_0',\dots,L_{v-2}',L_v\dots, L_n)\to 0\,\label{eq:Fvext1}\\
 0\to L_{v}(-x_{v+1})\to F^v\to \cL_{v}(L_0',\dots,L_{v-1}', L_{v+1}\dots, L_n)\to 0\,,\label{eq:Fvext2}
\end{align}
where $\cL_{v}(L_0',\dots,L_{v-1}', L_{v+1}\dots, L_n)$ denotes the unique line bundle on $C^v=C-C_v$ whose restriction to $C_u$ is $L_u'$ for $u=0,\dots,v-1$ and $L_u$ for $u=v+1,\dots,n$. In other words, $\cL_v(\_)$ is the inverse of the isomorphism $\pi_v$ of \autoref{lem:PicCo}. For $v=0$, only \eqref{eq:Fvext2} is valid. For $v=n+1$, only \eqref{eq:Fvext1} is valid.

By \eqref{eq:L'}, the outer terms of the short exact sequence \eqref{eq:Fvext2} for $v=u$ and of the short exact sequence \eqref{eq:Fvext1} for $v=u+1$ agree. This means that $[F^v]$ and $[F^{v+1}]$ are two points of the same $\IP^1$-fibre of $Y_v\to J_v$. More precisely, they are the intersection points of this one fibre of $Y_v$ with the components $Y_{v-1}$ and $Y_{v+1}$, respectively. Hence, we see that, when starting in $[F^0]\in Y_{n}\cap Y_{0}$ and walking the first lap of the characteristic cycle containing this point, we pass through $[F^1]$, $[F^2]$, \dots , $[F^n]$ until we reach $[F^{n+1}]$, which is again a point in $Y_{n}\cap Y_{0}$. However, $[F^{n+1}]\neq [F^0]$. More precisely, considering them both as line bundles on $C^{x_0}$, where line bundles are determined by their restrictions to irreducible components by \autoref{lem:PicCx}, we see that $[F^{n+1}]$ corresponds to $(L'_0,\dots, L'_n)\in J:=\prod_{v=0}^n\Pic_{b_v}(C_v)$ while $[F^0]$ corresponds to
$(L_0,\dots, L_n)\in J$. Hence, we get the following
\begin{prop}\label{prop:Ancycle}
Let $C$ be an extended ADE curve with intersection graph $\Gamma=\tilde A_n$. Then, the characteristic cycles of $\cM$ are configurations of smooth rational curves of type $\tilde A_{kn}$ where $k$ is the (possibly infinite) order of the translation automorphism
\[
 J\xrightarrow\cong J\quad\quad (L_0,\dots, L_n)\mapsto (L'_0,\dots, L'_n)= \bigl(L_0(x_0-x_1),\dots, L_n(x_n-x_{n+1})\bigr)\,.
\]
\end{prop}

The case $k=1$ only occurs if all $C_u$ are rational, i.e.\ if $C$ is a singular elliptic fibre.
This means that, in all other cases, the characteristic cycle can never close after just one lap. 
Indeed, for $g(C_u)\ge 1$, the line bundle $\reg_{C_u}(x_u-x_{u+1})$ is never trivial.

All the cases $k\ge 2$ actually occur, including the case $k=\infty$, as shown by the following example. 
 Let $ L\subset \IP^2$ be a line in the plane and fix four distinct points $p_1, \dots, p_4\in L$.
 Let $f\colon E\to L$ be the double cover branched over these four points. Choosing, say, the preimage of $p_1$ as the origin, $E$ is an elliptic curve, with multiplication by $-1$, which we denote by $\ominus$,  interchanging the points of the fibres of $f$. 
 
 Let now $s\in E$ be a point that is not $2$-torsion, which is equivalent to $f(s)\notin\{p_1,p_2,p_3,p_4\}$.  
 Counting parameters, we can choose a sextic $B\subset \IP^2$ which has a node at $f(s)$, is smooth away from $f(s)$, and intersects $L$ transversally in $p_1, \dots, p_4$. 
Now, as in \autoref{ex:doublecov}, let $X\to \IP^2$ be the minimal resolution of the double cover branched over $B$. Then the preimage $C$ of $L$ is an extended ADE curve with intersection graph $\Gamma=\tilde A_1$. One component is the exceptional divisor, hence rational, and the other component is $E$. The two intersection points of the components are $s$ and $\ominus s$. 
Hence, fixing some $\chi>0$ and some polarisation $H$ on $C$ satisfying \autoref{ass:H}, by \autoref{prop:Ancycle}, the characteristic cycle on the moduli space $\cM$ is of type $\tilde A_{kn}$ where $k$ is the order of the line bundle $\reg_E(s-\ominus s)$. Now, if $s$ is not torsion, then  
$\reg_E(s-\ominus s)$ is not torsion. If $s$ is $m$-torsion for $m$ odd, then $\reg_E(s-\ominus s)$ is $m$-torsion. Finally, if $s$ is $m$-torsion for $m$ even, then $\reg_E(s-\ominus s)$ is $\frac m2$-torsion. Hence, for any $k\ge 2$, we get characteristic cycles of type $\tilde A_{kn}$ by this construction.

\subsection{Relation to work of L\'{o}pez-Mart\'{\i}n, and other polarisations}\label{subsect:LM}

In \cite{LM--Simpson}, L\'{o}pez-Mart\'{\i}n classifies stable sheaves of type $m$ on $\tilde A_n$-configurations of smooth rational curves (i.e.\ extended ADE curves with $\Gamma=\tilde A_n$ and $C_o\cong \IP^1$ for every $o\in O$), and also on the other two types of reducible but reduced singular fibres of elliptic fibrations, namely types III and IV in Kodaira's list. Remarkably, this is done for every polarisation, not only those satisfying \autoref{ass:H}.

We believe that the results of \cite{LM--Simpson} generalise in a straight--forward way to arbitrary ADE curves of type $\Gamma=\tilde A_n$ with almost unchanged proofs. Furthermore, we believe that the classifications for types III and IV can be generalised to configurations of curves whose components intersect as in types III and IV, but may be of higher genera.

Anyway, even if we restrict ourselves to the case that all components of our extended ADE curve of type $\tilde A_n$ are rational, the results of \cite{LM--Simpson} show that \autoref{ass:H} is strictly necessary for our classification result \autoref{thm: easy classification} to hold. It is easy to see that we have properly semi-stable sheaves as soon as some $\frac{e_o}e \chi$ is an integer for some $o\in O$, i.e.\ if $b_0= \frac{e_o}e \chi$. But even if we choose our $H$ in a way that avoids properly semi-stable sheaves, the class of stable sheaves changes if $\sum_{o\in O}b_o$ has another value than the one required by \autoref{ass:H}.

For example, let $\Gamma=\tilde A_4$, $\chi=2$, and choose $H$ such that $e_o=1$ for all $o=0,1,2,3,4$. Then, $b_o=1$ for all $o\in O$, hence $\sum_{o\in O}b_o=5=\chi+|O|-2$. In this case, the stable line bundles $L_o$ on $C$ are those with
$\chi(L_o)=b_o$ for three of the components, and $\chi(L_o)=b_o+1$ for two of the components, and these two components are not allowed to intersect; compare \cite[Prop.\ 5.13]{LM--Simpson}. However, though the class of stable sheaves changes, one can still show in this case that the moduli space of stable sheaves of type $m$ is again a $\tilde A_4$-configuration of smooth rational curves.

\subsection{Negative $\chi$}

All our results are formulated for $\chi>0$. However, one can easily deduce analogous results for $\chi<0$ from them by duality. Indeed,  $F$ is a stable sheaf with Euler characteristic $\chi$ if and only if $F^D:=\sHom(F,\omega_C)$ is a stable sheaf of Euler characteristic $-\chi$. Hence, duality $(\_)^D$ gives an isomorphism between the moduli space of sheaves of type $m$ and Euler characteristic $\chi$ and the moduli space of sheaves of type $m$ and Euler characteristic $-\chi$.

One can also deduce a characterisation of the stable sheaves of type $m$ and negative Euler characteristic $-\chi$ from \autoref{thm: easy classification} using duality. Namely, under \autoref{ass:H} the stable sheaves are exactly those of the form $F\cong L(-x):=L\otimes \cI_{x\hookrightarrow C}$ for some $x\in C$ and some line bundle $L\in \Pic(C)$ with $\chi(L_{C_v})=-b_v+2$.

Note, however, that for $\chi=0$ properly semi-stable sheaves occur, and the description of the moduli space changes. For example, for $\Gamma=\tilde A_n$, the moduli space becomes just one $\IP^1$-bundle over $J$ which is glued with itself along two sections; compare \cite[Thm.\ 4.1(2)]{LM--relative} (again, the statement in \emph{loc.\ cit.}\ is for singular elliptic fibres, but can be extended to our more general set-up of extended ADE curves).

\bibliographystyle{alpha}
\addcontentsline{toc}{chapter}{References}
\bibliography{references}

\end{document}